\documentclass[12pt,a4paper,twoside,final,notitlepage, reqno]{article}

\usepackage[latin1]{inputenc}
\usepackage[english,french]{babel}
\usepackage[a4paper,left=2.5cm,right=2.5cm]{geometry}
\usepackage{amssymb}
\usepackage{amsthm} 
\usepackage{amsmath}
\usepackage{graphics,graphicx}
\usepackage{epstopdf}
\usepackage[usenames, dvipsnames]{color}
\usepackage[textsize=small]{todonotes}
\usepackage{booktabs}
\usepackage{subfigure} 
\usepackage{siunitx}

\graphicspath{{figs/}}

\makeatletter
\def\input@path{{figs/}}
\makeatother

\theoremstyle{definition}
\newtheorem{theorem}{Theorem}
\newtheorem{lemma}{Lemma}
\newtheorem{proposition}{Proposition}
\newtheorem{corollary}{Corollary}

\theoremstyle{definition}
\newtheorem{definition}{Definition}

\theoremstyle{definition}

\newtheorem*{lemme-4}{Lemma 4}

\providecommand{\keywords}[1]  {\textbf{Keywords:} #1}
\providecommand{\subjclass}[1] {\textbf{Subject classification:} #1}

\usepackage{esdiff}

\let\div\undefined
\DeclareMathOperator{\div}{div} 

\DeclareMathOperator{\Span}{span}

\DeclareMathOperator{\supp}{supp}

\newcommand{\R}{\mathbb{R}}

\renewcommand{\phi}{\varphi}
\renewcommand{\epsilon}{\varepsilon}

\usepackage{amscd}
\DeclareMathOperator{\cotan}{cotan} 
\DeclareMathOperator{\gT}{\nabla_{\T}} 
\DeclareMathOperator{\tr}{tr} 

\newcommand{\Hdiv}{{\rm H}(\div,\Omega)}
\newcommand{\ld}  {{\rm L}^2(\Omega)} 
\newcommand{\huo} {{\rm H}^1_0(\Omega)} 
\newcommand{\ldd} {\left [ \ld \right]^2} 
\newcommand{\dd} {\, {\rm d}}
\newcommand{\vphi}{\varphi}
\newcommand{\vphis}{\varphi^\star}
\newcommand{\vphia}{\varphi_a}
\newcommand{\vphias}{\varphi_a^\star}
\newcommand{\vphisKi}{\varphi_{K,i}^\star}
\newcommand{\vphisLj}{\varphi_{L,j}^\star}
\newcommand{\vphiaK}{\varphi_{K,a}}
\newcommand{\vphiaL}{\varphi_{L,a}}

\newcommand{\vphiasK}{\varphi_{K,a}^\star}
\newcommand{\vphiasL}{\varphi_{L,a}^\star}
\newcommand  {\un}  {1\hspace{-0.09cm} {\rm l}}

\newcommand  {\ai}  {a\in \T^1_i, ~ \partial^c a = (K,L)}
\newcommand  {\ab}  {a\in \T^1_b, ~ \partial^c a =(K)}
\newcommand  {\T}   {\mathcal{T}}

\newcommand  {\drtb}{ \left( \vphias \right)_{a\in \T^1}}

\newcommand  {\ca}{ ( \vphias,\vphia)_0}

\def\ib#1{_{_{\scriptstyle{#1}}}}

\newcommand  {\di}    {\displaystyle}

\newcommand{\monitem}{ \smallskip \noindent $\bullet$ \quad  } 
\newcommand{\moneq}{\vspace*{-6pt} \begin{equation} \centering \displaystyle } 
  \newcommand{\moneqstar}{\vspace*{-6pt} \begin{equation*} \centering \displaystyle } 
    \newcommand{\monendstar}{\vspace*{-6pt} \end{equation*}   }
  \newcommand{\monend}{\vspace*{-6pt} \end{equation}   }

\usepackage{authblk}

\title{
  \vspace{-.5cm}
  \bf \Large{
    Raviart-Thomas finite elements 
    \\[7pt] 
    of  Petrov-Galerkin type
  }}

\author[1,2]{Fran\c cois  Dubois \thanks{francois.dubois@u-psud.fr}}
\author[3]{Isabelle Greff \thanks{isabelle.greff@univ-pau.fr}}
\author[3]{Charles Pierre \thanks{charles.pierre@univ-pau.fr}}

\affil[1]{Laboratoire  de Math\'ematiques, Universit\'e  Paris-Sud, Orsay.}
\affil[2]{Conservatoire National des Arts et M\'etiers. 
  LMSSC, Paris}
\affil[3]{
  Laboratoire  de Math\'ematiques et de leurs Applications, 
  UMR CNRS 5142 \protect \\
  Universit\'e de Pau et des Pays de l'Adour, France.}

\usepackage{fancyhdr}

\usepackage[title,titletoc]{appendix}

\begin{document}
\selectlanguage{english}

\date{1 February,  2019. }  

\maketitle
\vspace{.3cm}
\begin{abstract}
  \noindent 
    Finite volume methods are widely used, in particular because they allow an explicit and local computation of a discrete gradient. This computation is only based on the values of a given scalar field. In this contribution, we wish to achieve the same goal in a mixed finite element context of Petrov-Galerkin type so as to ensure a  local computation of the gradient at the
  interfaces of the elements.
  The shape functions are the Raviart-Thomas finite elements. 
  Our purpose
  is to define test functions that are in duality with these shape functions:
  precisely, the shape and test functions will be asked to satisfy 
  some orthogonality property.
This paradigm is addressed for the discrete solution of the Poisson problem.
  The general theory of  Babu\v{s}ka 
  brings  necessary and sufficient stability conditions for a
  Petrov-Galerkin mixed problem   to be   convergent. 
  In order to ensure stability, we propose specific constraints for the  dual test functions.
  With this choice, we prove  that the
  mixed Petrov-Galerkin scheme  is identical to the
  four point finite volume scheme of Herbin,
  and to the  mass lumping approach  developed
  by Baranger, Maitre and Oudin. 
  Convergence is proven with the usual techniques of mixed finite elements.
\end{abstract}
\vspace{.3cm}
\noindent
\\[10pt] \noindent
\\[10pt] 

%
\selectlanguage{french}
\begin{abstract}
  \noindent
    La m\'ethode des volumes finis est largement diffus\'ee, en particulier parce qu'elle permet un calcul local et explicite du gradient discret.
  Ce calcul ne s'effectue qu'\`a partir des valeurs donn\'ees d'un champ scalaire.
  L'objet de cette contribution est d'atteindre un but similaire dans le contexte des \'el\'ements finis mixtes \`a l'aide d'une formulation Petrov-Galerkin qui permet un
  calcul local du gradient aux interfaces des \'el\'ements.
  Il s'agit d'expliciter des fonctions test duales pour l'\'el\'ement fini de Raviart-Thomas~: 
  pr\'ecis\'ement les fonctions de forme et les fonctions test 
  doivent satisfaire une relation d'orthogonalit\'e.
  Ce paradigme est discut\'e pour la r\'esolution discr\`ete de l'\'equation de Poisson.
  La th\'eorie g\'en\'erale de Babu\v{s}ka 
  permet de garantir des conditions de stabilit\'e n\'ecessaires et suffisantes pour qu'un
  probl\`eme mixte  de Petrov-Galerkin
  conduise \`a une approximation convergente. 
  Nous proposons  des contraintes sp\'ecifiques sur 
  les fonctions test duales 
  afin  de garantir la stabilit\'e.  
  Avec ce choix, nous montrons   que le sch\'ema
  mixte de Petrov-Galerkin obtenu est identique au sch\'ema de volumes finis \`a quatre points de Herbin
  et \`a l'approche par condensation de masse d\'evelopp\'ee
  par Baranger, Maitre et Oudin.  
  Nous montrons  enfin la convergence
  avec les m\'ethodes usuelles d'\'el\'ements finis mixtes.
\end{abstract}
\selectlanguage{english}

\section*{Introduction}

Finite volume methods are very popular for the approximation of conservation laws.
The unknowns are mean values of conserved quantities in a given family of cells, 
also named ``control volumes''. 
These mean values are linked together by numerical fluxes. 
The fluxes are defined and computed on interfaces between two control volumes. 
They are  defined with the help of cell values on each side of the interface. 
For hyperbolic problems, the computation of fluxes is obtained  by linear or nonlinear interpolation 
(see {\it e.g.} Godunov {\it et al.} \cite {Godunov79}).

This  paper  addresses  the question of flux computation for second order elliptic problems.
To fix the ideas, we restrict ourselves to the Laplace operator. 
The computation of flux is held by differentiation: the interface flux must be an approximation 
of the normal derivative of the unknown function at the interface between two control volumes. 
The computation of diffusive fluxes using finite difference formulas on the mesh interfaces has been addressed by much research for more than 50 years, as detailed below.
Observe that  for problems involving both advection and
diffusion, the method of Spalding and Patankar \cite{Patankar80} defines
a combination of interpolation for the advective part and derivation for the diffusive part.  

The well known two point flux approximation (see Faille, Gallou\"et and
Herbin \cite{FGH91, Herbin95})
is based on a finite difference formula
applied to two scalar unknowns on each side of the interface.
These unknowns are ordered in the normal direction of the interface 
considering a Voronoi dual mesh of the original mesh, \cite{Voronoi1908}.
When the mesh does not satisfy the Voronoi condition, 
the normal direction of the interface does not coincide with the direction of the 
centres of the cells. The tangential component of the gradient needs
to be introduced. 
We refer to the  ``diamond scheme'' proposed by Noh in \cite{Noh} in 1964 
for triangular meshes and analysed by 
Coudi\`ere, Vila and Villedieu \cite{CVV}. 
The computation of diffusive interface gradients for hexahedral meshes
was studied by 
Kershaw \cite{Kershaw81}, Pert \cite{Pert81} and Faille \cite{Fa92}.
An extension of the finite volume method with
duality between cells and vertices 
has also been proposed by Hermeline \cite{Hermeline} and 
Domelevo and Omnes \cite{Dom-Omnes}. 

The finite volume method has been originally proposed as a numerical method
in engineering
\cite{Patankar80, Rivas82}.
Eymard {\it et  al.} (see {\it e.g.} \cite{gallouet})
proposed a mathematical framework for the analysis of finite volume
methods based on a discrete functional approach.
Even if the method is non consistent
in the sense of finite differences, 
they proved convergence. 
Nevertheless, 
a natural question is the reconstruction of a discrete gradient
from the interface fluxes. 
This question has been first considered for interfaces
with normal direction different to the direction of
the neighbour nodes
by
\cite{Noh, Kershaw81, Pert81, Fa92}. From a 
mathematical point of view, a natural condition is
the existence of the divergence
of the discrete gradient:
how to impose the condition that the discrete gradient 
belongs to the space ${\rm H}(\div)$ ?
If this mathematical condition is satisfied, it is natural to consider
mixed formulations.
After the pioneering work of Fraeijs de Veubeke \cite{FdVeubeke65},
mixed finite elements for two-dimensional space
were introduced
by Raviart and Thomas \cite{RT77} in 1977. They will be denoted as ``RT'' finite elements 
in this contribution.

The discrete gradient built from the RT mixed finite element
is non local.
  Precisely, this discrete gradient for the mixed finite 
element method 
  of a scalar shape function $u$ is defined as the unique $p:=\nabla_h u\in\,$RT
  so that $(p,q)_0=-(u,\div q)_0$ for all $q\in\,$RT.
  With this definition, 
  the flux component of $p$ for a given mesh interface
  cannot be computed locally using only the values of $u$ in the interface neighbourhood.
This is not suitable for the discretisation 
of a differentiation operator that is essentially local. 
In their contribution \cite{BMO96}, Baranger, Maitre and Oudin proposed
a  mass lumping of the RT mass matrix
to overcome this difficulty.
They introduced an appropriate quadrature rule to approximate the exact mass matrix.
With this approach, the interface flux 
is reduced to  a true two-point formula.
Following the  idea in \cite{BMO96}, for general diffusion problems,
  further works have investigated the relationships between local flux expressions and mixed finite element methods.
  Arbogast, Wheeler and Yotov in \cite {arbogast-1997}
  present a variant of the classical mixed finite element method
  (named expanded mixed  finite element).
  They shown that,
  in the case of the lowest order Raviart-Thomas elements on rectangular meshes,
  the approximation of the expanded mixed finite element method
  using a specific quadrature rule leads to  a cell-centered scheme on the scalar unknown. 
  That scheme involves local flux expressions based on finite difference rules.
  The results in \cite{arbogast-1997} were extended by
  Wheeler and Yotov
  in \cite{wheeler-2006}  for the classical  mixed finite element method. 
  The multipoint flux approximation methods propose to evaluate local fluxes with finite difference formula, see \textit{e.g.} Aavatsmark  \cite{aavatsmark-2002}.
  That method has been later shown in Aavatsmark \textit{et al.}  \cite{aavatsmark-2007} to be equivalent on quadrangular meshes with the mixed finite element method with low order elements implemented with a specific quadrature rule.
  Local flux computation using the 
  Raviart-Thomas basis functions has also been developed by Youn\`es \textit{et al.} 
  in \cite{Younes-Chavent-1999}.
  That question has been further investigated by Vohral\'{i}k in \cite{vohralik-2006}. He shows that the mixed finite element discrete gradient
  $p=\nabla_h u$ can be computed locally with the help both of $u$ and of the source term $f_h := -\div(\nabla _h u)$ (that depends on $\nabla _h u$).
More precisely, with a slight modification of the discrete source term $f_h$ in the finite volume method, it has been proven 
in \cite{chavent-2003,vohralik-wohlmuth-2013}
that the two discrete gradients defined either with the mixed finite element or with the finite volume method are identical.
Moreover, in case of a vanishing source term $f=0$, the two discrete gradients are identical without any modification of the discrete source term $f_h$ 
(see \cite{Younes-Chavent-1999}, \cite{chavent-2003}, \cite{younes-2004}).

Our purpose is to build a discrete gradient with a local computation on the mesh interfaces, 
that is conformal in H(div).
Our paradigm is to define this discrete gradient 
only using the scalar field and
without considering the source term.
On the contrary of the previously discussed works
  \cite{BMO96,arbogast-1997,aavatsmark-2007,wheeler-2006}, the expression of that discrete gradient will not be obtained through an approximation of a discrete mixed  problem using quadrature rules. It will be obtained from the variational setting itself.
The main idea is to choose a test function space that is 
L$^2$-orthogonal
with the shape functions, \textit{i.e.} in duality with the Raviart-Thomas space.
With a Petrov-Galerkin approach the spaces of the shape  and test functions are different.
It is now possible to insert  duality between the shape and test functions and then to 
recover a local definition of the discrete gradient, as we proposed previously 
in the one-dimensional case \cite {D00}.
The stability analysis of the mixed finite element method emphasises the 
``inf-sup''  condition \cite{lady69, Ba71, Brezzi74}. 
In his fundamental contribution, Babu{\v{s}}ka \cite{Ba71} gives general inf-sup conditions
for mixed Petrov-Galerkin (introduced in \cite{Petrov}) formulation.
The inf-sup condition guides the construction of the dual space.

In this contribution we extend the Petrov-Galerkin
formulation to two-dimensional space
dimension with Raviart-Thomas shape functions.
In section \ref{section:notations},
we introduce notations and general backgrounds.
The  discrete gradient is presented in section
\ref{section:discgrad}.
Dual  Raviart-Thomas test functions for the  Petrov-Galerkin formulation
of Poisson equation
are proposed in section \ref{section:dualscheme}.
In section \ref{sec:recover-VF4}, we retrieve the four point finite volume scheme
of Herbin \cite{Herbin95}
for a specific choice of the dual test functions.
Section \ref{sec:stab-conv} is devoted to the stability
and convergence analysis in Sobolev spaces with standard finite element methods.
\section{Background and notations}
\label{section:notations}
In the sequel,  $\Omega\subset \R^2$ is an open bounded convex 
with a polygonal boundary.
The spaces $\ld$, $\huo$ and $\Hdiv$ are considered, see \textit{e.g.} \cite{RT-book}.
The  ${\rm L}^2$-scalar products on $\ld$ and on $\ldd$ are similarly denoted  $(\cdot,\cdot)_0$.

\subsubsection*{Meshes}
\label{sec:mesh}
\begin{figure}[!ht]
  \begin{center}
    \begin{picture}(0,0)%
      \includegraphics{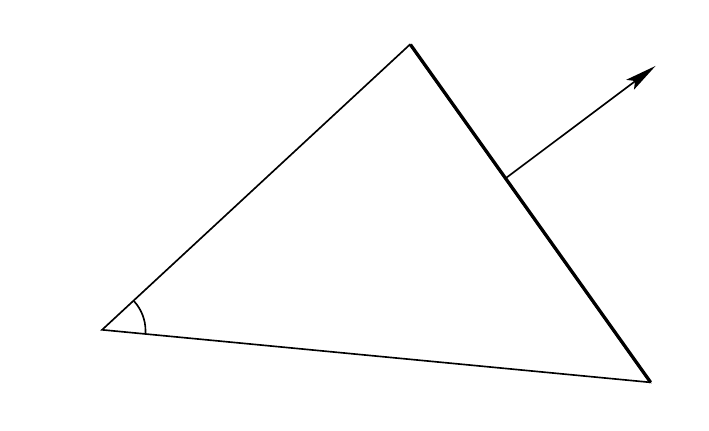}%
    \end{picture}%
    \setlength{\unitlength}{4144sp}%
    \begingroup\makeatletter\ifx\SetFigFont\undefined%
    \gdef\SetFigFont#1#2#3#4#5{%
      \reset@font\fontsize{#1}{#2pt}%
      \fontfamily{#3}\fontseries{#4}\fontshape{#5}%
      \selectfont}%
    \fi\endgroup%
    \begin{picture}(3226,1967)(278,-1228)
      \put(1818,-419){\makebox(0,0)[lb]{\smash{{\SetFigFont{10}{12.0}{\familydefault}{\mddefault}{\updefault}{\color[rgb]{0,0,0}{\large $K$}}%
            }}}}
      \put(1021,-734){\makebox(0,0)[lb]{\smash{{\SetFigFont{10}{12.0}{\familydefault}{\mddefault}{\updefault}{\color[rgb]{0,0,0}$\theta_{K,i}$}%
            }}}}
      \put(3331,434){\makebox(0,0)[lb]{\smash{{\SetFigFont{10}{12.0}{\familydefault}{\mddefault}{\updefault}{\color[rgb]{0,0,0}$n_{K,i}$}%
            }}}}
      \put(451,-1006){\makebox(0,0)[lb]{\smash{{\SetFigFont{10}{12.0}{\familydefault}{\mddefault}{\updefault}{\color[rgb]{0,0,0}$W_{K,i}$}%
            }}}}
      \put(2386,-376){\makebox(0,0)[lb]{\smash{{\SetFigFont{10}{12.0}{\familydefault}{\mddefault}{\updefault}{\color[rgb]{0,0,0}$a_{K,i}$}%
            }}}}
    \end{picture}%
  \end{center}
  \caption{Mesh notations for a triangle $K\in\T^2$}
  \label{fig:cell-1}
\end{figure}
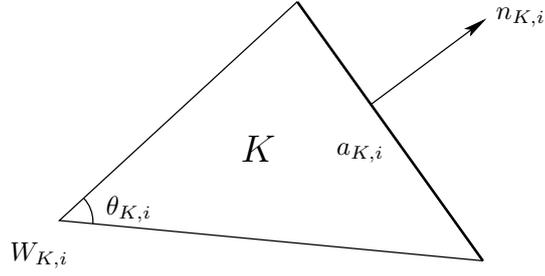
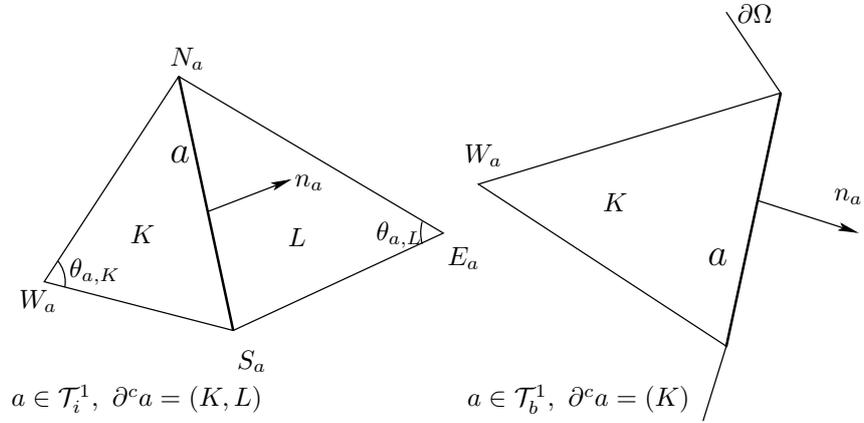
\begin{figure}[!ht]
  \begin{center}
    \begin{picture}(0,0)%
      \includegraphics{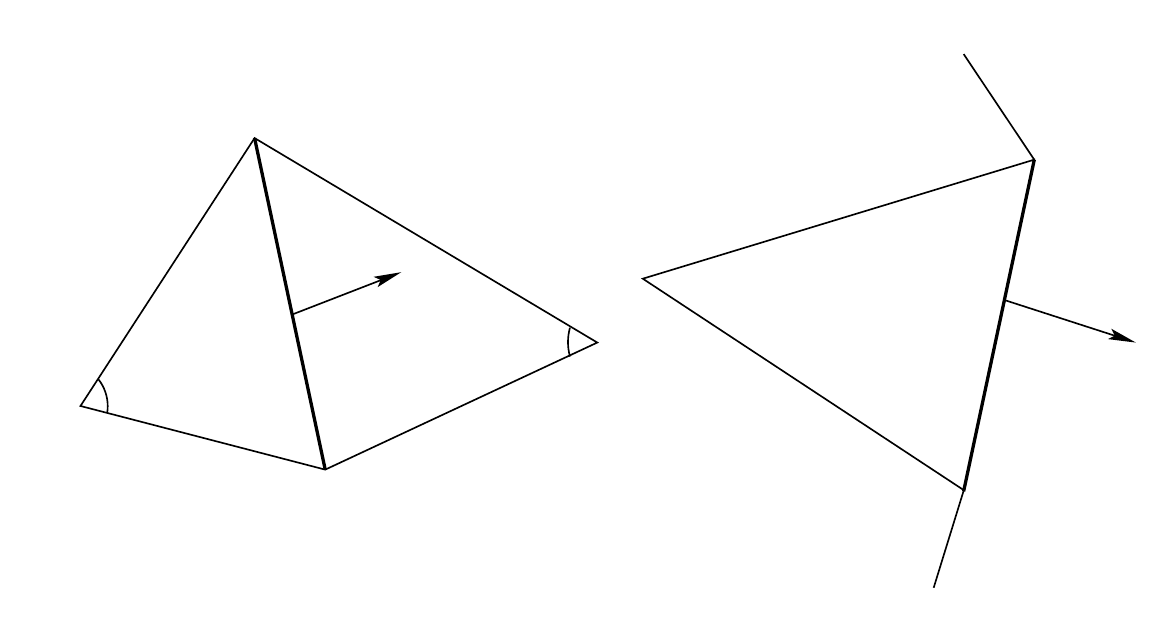}%
    \end{picture}%
    \setlength{\unitlength}{4144sp}%
    \begingroup\makeatletter\ifx\SetFigFont\undefined%
    \gdef\SetFigFont#1#2#3#4#5{%
      \reset@font\fontsize{#1}{#2pt}%
      \fontfamily{#3}\fontseries{#4}\fontshape{#5}%
      \selectfont}%
    \fi\endgroup%
    \begin{picture}(5363,2853)(1988,-1896)
      \put(2504,-869){\makebox(0,0)[lb]{\smash{{\SetFigFont{10}{12.0}{\familydefault}{\mddefault}{\updefault}{\color[rgb]{0,0,0}$\theta_{a,K}$}%
            }}}}
      \put(3839,-327){\makebox(0,0)[lb]{\smash{{\SetFigFont{10}{12.0}{\familydefault}{\mddefault}{\updefault}{\color[rgb]{0,0,0}$n_a$}%
            }}}}
      \put(2866,-667){\makebox(0,0)[lb]{\smash{{\SetFigFont{10}{12.0}{\familydefault}{\mddefault}{\updefault}{\color[rgb]{0,0,0}$K$}%
            }}}}
      \put(3800,-685){\makebox(0,0)[lb]{\smash{{\SetFigFont{10}{12.0}{\familydefault}{\mddefault}{\updefault}{\color[rgb]{0,0,0}$L$}%
            }}}}
      \put(3101,-186){\makebox(0,0)[lb]{\smash{{\SetFigFont{10}{12.0}{\familydefault}{\mddefault}{\updefault}{\color[rgb]{0,0,0}{\large $a$ }}%
            }}}}
      \put(6285,-799){\makebox(0,0)[lb]{\smash{{\SetFigFont{10}{12.0}{\familydefault}{\mddefault}{\updefault}{\color[rgb]{0,0,0}{\large $a$}}%
            }}}}
      \put(7032,-412){\makebox(0,0)[lb]{\smash{{\SetFigFont{10}{12.0}{\familydefault}{\mddefault}{\updefault}{\color[rgb]{0,0,0}$n_a$}%
            }}}}
      \put(5664,-476){\makebox(0,0)[lb]{\smash{{\SetFigFont{10}{12.0}{\familydefault}{\mddefault}{\updefault}{\color[rgb]{0,0,0}$K$}%
            }}}}
      \put(3101,394){\makebox(0,0)[lb]{\smash{{\SetFigFont{10}{12.0}{\familydefault}{\mddefault}{\updefault}{\color[rgb]{0,0,0}$N_a$}%
            }}}}
      \put(3499,-1412){\makebox(0,0)[lb]{\smash{{\SetFigFont{10}{12.0}{\familydefault}{\mddefault}{\updefault}{\color[rgb]{0,0,0}$S_a$}%
            }}}}
      \put(4843,-186){\makebox(0,0)[lb]{\smash{{\SetFigFont{10}{12.0}{\familydefault}{\mddefault}{\updefault}{\color[rgb]{0,0,0}$W_a$}%
            }}}}
      \put(4743,-799){\makebox(0,0)[lb]{\smash{{\SetFigFont{10}{12.0}{\familydefault}{\mddefault}{\updefault}{\color[rgb]{0,0,0}$E_a$}%
            }}}}
      \put(6460,652){\makebox(0,0)[lb]{\smash{{\SetFigFont{10}{12.0}{\familydefault}{\mddefault}{\updefault}{\color[rgb]{0,0,0}$\partial \Omega$}%
            }}}}
      \put(2161,-1636){\makebox(0,0)[lb]{\smash{{\SetFigFont{10}{12.0}{\familydefault}{\mddefault}{\updefault}{\color[rgb]{0,0,0} $\ai$}%
            }}}}
      \put(2206,-1051){\makebox(0,0)[lb]{\smash{{\SetFigFont{10}{12.0}{\familydefault}{\mddefault}{\updefault}{\color[rgb]{0,0,0}$W_a$}%
            }}}}
      \put(4321,-646){\makebox(0,0)[lb]{\smash{{\SetFigFont{10}{12.0}{\familydefault}{\mddefault}{\updefault}{\color[rgb]{0,0,0}$\theta_{a,L}$}%
            }}}}
      \put(4861,-1636){\makebox(0,0)[lb]{\smash{{\SetFigFont{10}{12.0}{\familydefault}{\mddefault}{\updefault}{\color[rgb]{0,0,0} $\ab$}%
            }}}}
    \end{picture}%
  \end{center}
  \caption{Mesh notations for an internal edge (left) and for a boundary edge (right)}
  \label{fig:edge}
\end{figure}
A conformal triangle mesh $\T$ of $\Omega$ is considered, in the sense of 
Ciarlet in \cite{Ciarlet}.
The angle, vertex, edge and triangle sets of $\T$ are respectively denoted 
{$\T^{-1}$,}
$\T^{0}$, $\T^1$ and $\T^2$.  
The area of  $K\in\T^2$ and the length of $a\in\T^1$
are denoted $\vert  K\vert $ and   $\vert  a\vert $.

Let $K\in\T^2$.
Its three edges, vertexes and angles  are respectively denoted $a_{K,i}$, $W_{K,i}$ and  $\theta_{K,i}$, (for $1 \leq i \leq 3$)
in such a way that $W_{K,i}$ and  $\theta_{K,i}$ are opposite to $a_{K,i}$
(see figure \ref{fig:cell-1}). 
The unit normal to $a_{K,i}$ pointing outwards $K$ is denoted $n_{K,i}$.
The local scalar products on $K$ are introduced as,
for $f_i\in\ld$ or  $p_i\in\ldd$:
\begin{displaymath}
  (f_1,f_2)_{0,K} = \int_K f_1f_2 \dd x \quad {\rm or} \quad 
  (p_1,p_2)_{0,K} = \int_K p_1\cdot p_2 \dd x\,.
\end{displaymath}

Let $a\in\T^1$.
One of its two unit normal is chosen and denoted $n_a$. This sets  an orientation for $a$.
Let $S_a,~ N_a$ be the two vertexes of $a$, ordered so that $(n_a, S_aN_a)$ has a direct orientation.
The sets $\T^1_i$ and $\T^1_b$ of the internal and boundary edges respectively are defined as,
\begin{displaymath}
  \T^1_b = \left\{ a\in\T^1, \quad  a\subset \partial\Omega\right\},\quad 
  \T^1_i = \T^1 \backslash \T^1_b.
\end{displaymath}
Let $a\in\T^1_i$. Its coboundary $\partial^c a$ is made of the unique ordered pair $K$, $L\in\T^2$ so that $a\subset  \partial K \cap \partial L$ and so that $n_a$ points from $K$ towards $L$. In such a case the following notation will be used:
\begin{displaymath}
  \ai
\end{displaymath}
and we will denote $W_a$ (\textit{resp.} $E_a$) the vertex of K (\textit{resp.} $L$) 
opposite to $a$ (see figure \ref{fig:edge}).
\\
Let $a\in\T^1_b$: $n_a$ is assumed to point towards the outside of $\Omega$. Its coboundary is made of a single $K\in\T^2$ so that $a\subset \partial K$, which situation is denoted as follows:
\begin{displaymath}
  \ab
\end{displaymath}
and we will denote $W_a$ the vertex of K opposite to $a$.
If $a\in\T^1$ is an edge of $K\in \T^2$, the angle of $K$ opposite to $a$ is denoted $\theta_{a,K}$.

\subsubsection*{Finite element spaces}
\label{sec:fem}
Relatively to a mesh $\T$ are defined the spaces $P^0$ and $RT$.
The space of piecewise constant functions on the mesh is denoted
by $P^0$ subspace of $\ld$.
The classical basis of $P^0$ is made of the indicators $\un_K$ for $K\in\T^2$.
To $u\in P^0$ is associated the vector $(u_K)_{K\in\T^2}$ so that
$   u = \sum_{K\in \T^2} u_K \, \un_K $.
The space of Raviart-Thomas of order 0  introduced in \cite{RT77} is denoted by $RT$ and is a subspace of $\Hdiv$.
It is recalled that $p \in RT$ if and only if $p \in \Hdiv$ and for all
$K\in \T^2$, 
$    p(x) = \alpha_K + \beta_K x$, for $x\in K$,
where $\alpha_K\in\R^2$ and $\beta_K\in \R$ are two constants.
An element $p\in RT$ is uniquely determined by its fluxes  
$ \, p_a := \int_a p\cdot n_a ds$ for $a\in\T^1$. 
The 
classical 
basis $\{\vphi_a, \, a \in \T^1\}$    
of $RT$ is so that 
$ \, \int_b \vphi_a\cdot n_b \dd s = \delta_{ab} \, $ 
for  all $b\in\T^1$ and 
with $\delta_{ab}$ the Kronecker symbol.
Then to $p\in RT$ is associated its flux vector 
$(p_a)_{a\in\T^1}$ so that,
$
p = \sum_{a\in\T^1} p_a \vphi_a
$.
\\
The \textit{local Raviart-Thomas basis functions} are defined, for $K\in\T^2$ and $i=1$, 2, 3, by:
\begin{equation}
  \label{eq:phi_K-i}
  \vphi_{K,i}(x) = \dfrac{1}{4\vert  K\vert } \nabla \vert x-W_{K,i}\vert ^2
  \quad {\rm on}\quad  K
  \quad  {\rm and} \quad  \vphi_{K,i} = 0 \quad {\rm otherwise}.
\end{equation}
With that definition:
\begin{equation}
  \label{eq:loc-basis-RT}
  \begin{array}{lll}
    \vphi_a = \vphi_{K,i} - \vphi_{L,j} \quad  &\text{if} \quad \ai
    \quad  &\text{and}\quad a=a_{K,i}=a_{L,j}
    \\[5pt]
    \vphi_a = \vphi_{K,i}\quad  &\text{if} \quad \ab
    \quad  &\text{and}\quad a=a_{K,i}
  \end{array}.
\end{equation}
The support of the RT basis functions is
$\supp(\vphi_a) =  K \cup L$ if $\ai$ or $\supp(\vphi_a) =  K$ in case $\ab$.
This provides a  second way to decompose $p\in RT$ as,
\begin{displaymath}
   p = \sum_{K\in\T^2} \sum_{i=1}^3 p_{K,i} \,\, \vphi_{K,i},
\end{displaymath}
where $p_{K,i} = \varepsilon p_a$ if $a=a_{K,i}$ with $\varepsilon = n_a \cdot n_{K,i} = \pm 1$.
For simplicity
we will denote $\vphiaK = \vphi_{K,i}$ for $a\in\T^1$ such that $a\subset \partial K$ and $a=a_{K,i}$.
The divergence operator $\div:~ RT \rightarrow P^0$ is given by,
\begin{equation}
  \label{eq:RT-div}
  \div p = \sum_{K\in\T^2} \left(\div p\right)_K \un_K,
  \qquad  
  \left(\div p\right)_K= \dfrac{1}{| K|} \sum_{1=1}^3 p_{K,i}.
\end{equation}

\section{Discrete gradient}
\label{section:discgrad}
The two unbounded operators,
$  \nabla:~  \ld \supset  \huo \rightarrow \ldd$ and     
$  \div:~  \ldd  \supset $ $ \Hdiv  \rightarrow  \ld$    
together satisfy the Green formula: for $u\in\huo$ and $p\in\Hdiv$:
$
(\nabla  u, p)_0 = -(u, \div p)_0
$.
Identifying $\ld$ and $\ldd$ with their topological dual spaces using the ${\rm L}^2$-scalar 
product yields the following property,
\begin{displaymath}
  \nabla  = -\div ^\star,
\end{displaymath}
that is a weak definition of the gradient on $\huo$.
\\
Consider a mesh of the domain and the associated spaces $P^0$ and $RT$ as defined in section \ref{section:notations}.
We want  to define a \textit{discrete gradient}: $\gT:~ P^0\rightarrow  RT$,
based on a similar weak formulation. 
Starting from the divergence operator 
$
\div:~  RT \rightarrow P^0,
$
one can 
define $
\div^\star:~  \left(P^0\right)' \rightarrow \left(RT\right)', 
$
between the algebraic dual spaces of $P^0$ and $RT$.
The classical basis for $P^0$ is orthogonal for the ${\rm L}^2$-scalar product.
Thus, $P^0$ is identified with its algebraic dual $\left(P^0\right)'$.   
On the contrary, the Raviart-Thomas basis 
$\left\{\vphi_a, \,  a \in \T^1\right\}$ of $RT$ 
is not orthogonal.
For this reason, a general identification process of $\left(RT\right)'$  
to a space 
$RT^\star=\Span\left( \vphi_a^\star, \quad  a\in\T^1 \right)$
is studied.
We want it to satisfy,
\begin{equation}
  \label{eq:RT*-0}
  \vphi_a^\star \in \Hdiv, \quad 
  (\vphi_a^\star,\vphi_a)_0 \neq 0,
\end{equation}
so that $RT^\star \subset \Hdiv$, together with
the orthogonality property,
\begin{equation}
  \label{eq:RT*}
  (\vphi_a^\star,\vphi_b)_0 = 0 \quad {\rm for} \quad a,b\in\T^1, \quad a\neq b.
\end{equation}
The discrete gradient is defined with the diagram,
\begin{equation}
  \label{eq:grad-discret}
  \begin{CD} 
    RT         & @>{\div}>>    & P^0    
    \\
    @ V{\Pi}VV & \hfill        & @VV{id}V
    \\
    RT^\star   & @ <<{\div^\star}< & P^0
  \end{CD}
  \quad, \qquad \gT = -\Pi^{-1}\circ \div^\star~~:~P^0\rightarrow RT,
\end{equation}
where $\Pi:~ RT\rightarrow  RT^\star$
is the projection defined
by $\Pi \vphi_a = \vphi_a^\star$ for any $a\in \T^1$.
\\

\noindent
Various choices for $RT^\star$ are possible.
The first choice is to set $RT^\star=RT$, and therefore to build 
$\left\{ \vphi_a^\star, \,  a\in\T^1 \right\}$   
with a Gram-Schmidt orthogonalisation process on the Raviart-Thomas basis.
Such a choice has an important drawback. The  dual base function $\vphi_a^\star$ does not conserve 
a support located around the edge $a$. The discrete gradient matrix will be a full matrix related 
with the Raviart-Thomas mass matrix inverse. 
This is not relevant 
with the definition of the original gradient operator that is local in space.
This choice corresponds to the classical mixed finite element discrete gradient 
that is known to be associated with a full matrix \cite{RT77}. 
In order to overcome  this problem, Baranger, Maitre and Oudin \cite{BMO96} have proposed to 
lump the mass matrix of the mixed finite element method. 
They obtain a discrete {\it local} gradient. 
Other methods have been proposed by Thomas-Trujillo \cite{TT1,TT2}, by
Noh \cite{Noh}, and analysed by Coudi\`ere, Vila and  Villedieu \cite{CVV}.
Another approach is to add unknowns at the vertices, as developed by  
Hermeline \cite{Hermeline} and  Domelevo-Omnes \cite{Dom-Omnes}. \\
A second choice, initially proposed by Dubois and co-workers  \cite{D00, D02, BDLPT05, Du10},  
is  investigated in this paper. The goal is to search for  a dual basis satisfying equation \eqref{eq:RT*-0} and
in addition to the orthogonality property \eqref{eq:RT*}, the localisation constraint,
\begin{equation}
  \label{eq:RT*-2}
  \forall ~a\in\T^1,\quad 
  \supp(\vphi^\star_a) \subset  \supp (\vphi_a),
\end{equation}
in order to impose locality to the discrete gradient. 
We observe that due to the $ \, {H(\div)} $-conformity, 
we have continuity of the normal component on  the boundary of the co-boundary of the edge $ \,a $:
\begin{equation}
  \vphias\cdot n_b=0  \quad {\text  if} \quad   a \neq b \in \T^1 . 
\end{equation}
With such a constraint \eqref{eq:RT*-2}  the discrete gradient of $u\in P^0$  will be defined on each edge $a\in \T^1$ only from the two values of $u$ on each side of $a$ (as detailed in proposition \ref{prop:disc-grad-prop-1}).
In this context it is no longer asked to have $\vphi^\star_a \in RT$ so that 
$RT \neq RT^\star$: 
thus, this is a Petrov-Galerkin discrete formalism.

\section{Raviart-Thomas dual basis }
\label{section:dualscheme}
\begin{definition} 
  \label{def:RT}
  $\drtb$ is said to be a Raviart-Thomas dual basis if it satisfies (\ref{eq:RT*-0}), the orthogonality condition \eqref{eq:RT*}, the localisation condition \eqref{eq:RT*-2} and the following \textit{flux normalisation} condition:
  \begin{equation}
    \label{eq:RT*-3}
    \forall ~a,~b\in\T^1, \quad 
    \int_b \vphi_a^\star \cdot n_b  \dd s = \delta_{ab}, 
  \end{equation}
  as for the Raviart-Thomas basis functions $\vphia$, see section \ref{section:notations}.
  \\
  In such a case, $RT^\star = \Span(\vphi_a^\star, \, a\in\T^1)$  
  is the associated Raviart-Thomas dual space, $\Pi:~\vphi_a\in RT \mapsto  \vphi_a^\star \in RT^\star$ the projection onto $RT^\star$  and $\gT=-\Pi^{-1}\div^\star:~P^0\rightarrow RT$ the associated discrete gradient, as described in diagram \eqref{eq:grad-discret}.
\end{definition}
\subsection{Computation of the discrete gradient} 
\begin{proposition} 
  \label{prop:disc-grad-prop-1}
  Let $\drtb$ be a Raviart-Thomas dual basis.
  The discrete gradient is given 
  for $u\in P^0$, by the
  relation $\gT u = \sum_{a\in \T^1} p_a \vphi_a $ with,
  \begin{equation}
    \label{eq:disc-grad-1}
    \begin{array}{cc}
      \text{if} \quad \ai, \quad &p_a = \dfrac{u_L- u_K}{(\vphi_a,\vphi_a^\star)_0}
      \\[15pt] 
      \text{if} \quad \ab, \quad &p_a = \dfrac{- u_K}{(\vphi_a,\vphi_a^\star)_0}    
    \end{array}
    .
  \end{equation}
\end{proposition}
The formulation of the discrete gradient only depends on the coefficients $(\vphi_a^\star,\vphi_a)_0$.
The discretisation of the Poisson equation (see  the next subsection)
also only depends on these coefficients.
\\
The result of the localisation condition (\ref{eq:RT*-2}) is, as expected, a local discrete gradient: its value on an edge $a\in \T^1$ only depends on the values of the scalar function $u$ on each sides of $a$.
\\
The discrete gradient on the external edges expresses a  
homogeneous Dirichlet boundary condition.
At the continuous level, the gradient defined on the domain $\huo$
is the adjoint of the divergence operator on the domain  $\Hdiv$.
That property is implicitly recovered at the discrete level.
This is consistent since the discrete gradient is the adjoint of the divergence on the domain $RT$.
\begin{proof}
  
  Condition \eqref{eq:RT*-3} leads to
  $\di
  \int_b \vphi_a^\star \cdot n_b  \dd s =  \int_b \vphi_a \cdot n_b  \dd s$,
  for any $ a,\,b \in \T^1$.
  Then the divergence theorem
  implies
  that
  \begin{displaymath}
    \forall ~p\in RT,\quad \forall ~K\in\T^2,\quad 
    \int_K \div p  \dd x = \int_K \div( \Pi p)  \dd x\,,
  \end{displaymath}
  and so proves
  \begin{equation}  
    \label{eq:RT*-3.2}
    \forall ~(u,p)\in P^0\times  RT,\quad 
    (\div p, u)_0 = (\div (\Pi p), u)_0.
  \end{equation}
  Let us 
  prove that,
  \begin{equation}
    \label{eq:RT*-1-3.1}
    \forall ~u\in P^0,\quad \forall ~q\in RT^\star,\quad 
    ( \gT u, q)_0 = - (u, \div q)_0 .
  \end{equation}
  From property (\ref{eq:RT*}) one can check that, 
  \begin{displaymath}
    \forall ~q_1,~q_2\in RT^{\star},\quad 
    (\Pi^{-1} q_1,q_2)_0 = ( q_1,\Pi^{-1} q_2)_0\,.
  \end{displaymath}
  Now consider $u\in P^0$ and  $q\in RT^\star$. We have with (\ref{eq:RT*-3.2}),
  \begin{displaymath}
    (u, \div q)_0 = (u, \div (\Pi^{-1} q))_0 = (\div^\star u, \Pi^{-1} q)_0
    = (\Pi^{-1}\bigl(\div^\star u\bigl),  q)_0,
  \end{displaymath}
  which gives \eqref{eq:RT*-1-3.1} by definition of the discrete gradient.
  \\ \indent
  We can now prove \eqref{eq:disc-grad-1}.
  Let $u\in P^0$ and $p=\gT u\in RT$ that we decompose as  $\gT u = \di\sum_{a\in\T^1} p_a \vphi_a$.
  For any $a\in \T^1$, with (\ref{eq:RT*}),
  \begin{equation*}
    \bigl( \gT u, \vphi_a^\star \bigl)_0 = p_a \bigl( \vphi_a, \vphi_a^\star \bigl)_0,
  \end{equation*}
  and meanwhile with equation \eqref{eq:RT*-1-3.1} and \eqref{eq:RT*-3.2} successively,
  \begin{displaymath}
    \bigl( \gT u, \vphi_a^\star \bigl)_0 =
    -\bigl( u, \div \vphi_a^\star \bigl)_0=
    -\bigl( u, \div \vphi_a \bigl)_0\,.
  \end{displaymath}
  Finally, 
  $\div \vphi_a$ is explicitly given by,
  \begin{align}
    \label{eq:div-phi_a}
    \text{if} \quad \ai&: \quad 
    \div \vphi_a =\dfrac{1}{|K|} \un_K - \dfrac{1}{|L|}\un_L,
    \\ \notag
    \text{if} \quad \ab&: \quad 
    \div \vphi_a = \dfrac{1}{|K|}\un_K.
  \end{align}
  This yields relations (\ref{eq:disc-grad-1}).
\end{proof}

\subsection{Petrov-Galerkin discretisation of the Poisson problem}
\label{sec:PG-disc}
Consider the following Poisson problem on $\Omega$,
\begin{equation}
  \label{eq:laplace}
  - \Delta u = f \in \ld,\quad u=0 \quad {\rm on} \quad \partial\Omega.
\end{equation}
Consider a mesh $\T$ and  a Raviart-Thomas dual basis $\drtb$
as in definition \ref{def:RT} leading to the space $RT^\star$.
Let us denote $V= P^0 \times RT$ and $V^\star= P^0 \times RT^\star$.
The  mixed
Petrov-Galerkin discretisation of equation \eqref{eq:laplace} is:
find $(u,p)\in V$ so that,
\begin{equation}
  \label{eq:PG-laplace}
  \forall ~ (v,q)\in V^\star,\quad 
  (p,q)_0 + (u,\div q)_0 = 0 
  \quad  \text{and} \quad 
  -(\div p, v)_0         = (f,v)_0 .
\end{equation}
The mixed  Petrov-Galerkin discrete problem \eqref{eq:PG-laplace} reformulates as:
find $(u,p)\in V$ so that,
\begin{displaymath}
  \forall ~ (v,q) \in V^\star, \quad 
  Z\bigl((u,p),(v,q)\bigl)=-(f,v)_0.
\end{displaymath}
where the bilinear  form $Z$ 
is defined for 
$(u,p)\in V$ and  $ (v,q)\in  V^\star$ by,
\begin{equation}
  \label{eqdef-Z}
  Z\bigl((u,p),(v,q)\bigl) \equiv (u,\div q)_0+(p,q)_0+(\div p,v)_0 .
\end{equation}
\begin{proposition} [Solution of the mixed discrete problem] 
  \label{prop:PG-disc}
  The pair $(u,p)\in V$ is a  solution of problem \eqref{eq:PG-laplace} if and only if 
  \begin{equation}
    \label{eq:FV-laplace}
    \gT u = p, \qquad  -\div( \gT u) = f_\T,
  \end{equation}
  where $f_\T\in P^0$ is the projection of $f$, defined by,
  \begin{displaymath}
    f_\T = \sum_{K\in \T^2} f_K \, \un_K,\qquad 
    f_K = \dfrac{1}{|K|} \int_K f \dd x.
  \end{displaymath}
  If $\,\, (\vphi_a, \vphi_a^\star)>0 \,\, $ for all $ \,\, a\in\T^1$,  
  then problem \eqref{eq:PG-laplace} has a  unique solution.
\end{proposition}
Proposition \ref{prop:PG-disc} shows an equivalence between the mixed Petrov-Galerkin discrete problem  \eqref{eq:PG-laplace} and the discrete problem (\ref{eq:FV-laplace}). Problem  (\ref{eq:FV-laplace}) actually is  a finite volume problem.
Precisely, with 
\eqref{eq:disc-grad-1},
it becomes: find $u\in P^0$ so that, for all $K\in \T^2$:
\begin{displaymath}
  \sum_{
    \substack{
      a\in\T^1_i,~\partial^c a = (K,L)
          }}
  \dfrac{u_L - u_K}{\ca} 
  +
  \sum_{
    a\in\T^1_b,~\partial^c a = (K)
  }
  \dfrac{ - u_K}{\ca} = |K| f_K.
\end{displaymath}
It is interesting to notice that this problem only involves
the coefficients $\ca$ that are going to be computed later.

\begin{proof}
  Let $u\in P^0$, denote $p=\gT u \in RT$ and assume that $\div p = f_T$. 
  Then using relation \eqref{eq:RT*-1-3.1}, equation
  \eqref{eq:PG-laplace} clearly holds.
  \\
  Conversely, consider $(u,p)\in V$ a solution of problem \eqref{eq:PG-laplace}. 
  Relation \eqref{eq:RT*-1-3.1} implies that $p=\gT u$,
  as a result, $-\div (\gT u ) = f_\T$.
  \\ \indent
  We assume that 
  $(\vphi_a, \vphi_a^\star)>0$ for all $a\in\T^1$
  and prove existence and uniqueness.
  It suffices to prove that $u=0$ is the unique solution when $f_\T = 0$. 
  In such a case, $\div( \gT u) = 0$, and using successively (\ref{eq:RT*-3.2}) and (\ref{eq:RT*-1-3.1}):
  \begin{align*}
    0 = 
    -(\div (\gT u) , u)_0 &= 
    -(\div (\Pi\gT u) , u)_0 \\&= 
    (\Pi\gT u , \gT u )_0  \\&= 
    \sum_{a\in\T^1} p_a^2 (\vphi_a,\vphi_a^\star)
    .
  \end{align*}
  As a result $p_a=0$ for all $a\in\T^1$ and $p=\gT u=0$.
  From (\ref{eq:PG-laplace}) it follows that for all $q\in RT^\star$ we have $(u,\div q)_0=0$.
  Thus with (\ref{eq:RT*-3.2}) we also have $(u,\div q)_0=0$ for all $q\in RT$.
  Since $\div(RT)=P^0$ it follows that $u=0$.
\end{proof}

%
%
\section{Retrieving the four point finite volume scheme}
\label{sec:recover-VF4}
In this section we present sufficient conditions for the
construction of Raviart-Thomas dual basis.
These conditions will allow to compute
the coefficients
$(\vphi_a^\star, \vphi_a)_0$.
We start by introducing the normal flux $g$ on the edges, and the divergence
of the dual basis $\delta_K$ on  $K\in \T^2$.
\\
Let $g:~(0,1)\rightarrow \R$ be a continuous function so that,
\begin{equation}
  \label{eq:def-g}
  \int_0^1 g  \dd s  = 1,\qquad
  \int_ 0^1 g(s)s^2  \dd s =0,
  \quad g(0) = 0 
  \quad \text{and}\quad
  g(s) = g(1-s). 
\end{equation}
On a mesh $\T$ are defined $g_{K,i}:~ a_{K,i} \rightarrow \R$ for $K\in\T^2$ and $i=1$, 2, 3 as,
\begin{equation}
  \label{eq:def-g_Ki}
  g_{K,i}(x) = g(s) / |a_{K,i}|\quad \text{for}\quad 
  x = s S_{K,i} + (1-s) N_{K,i}.
\end{equation}
For $K\in \T^2$ is denoted $\delta_K\in {\rm L}^2(K)$
a function that satisfies,
\begin{equation}
  \label{eq:def-delta_K}
  \int_K \delta_K \dd x = 1 \quad \text{and}\quad
  \int_K \delta_K(x) |x - W_{K,i}|^2 \dd x = 0 
  \quad \text{for}\quad
  i =1,~2,~ 3.
\end{equation}

To a family $(\vphisKi)$ of functions on $\Omega$ for $K\in\T^2$ and for $i=1$, 2, 3
is associated the family $(\vphias)_{a\in\T^1}$ so that,
\begin{equation}
  \label{eq:loc-basis}
  \begin{array}{lll}
    \vphias = \vphisKi - \vphisLj  \quad  &\text{if} \quad \ai
    \quad  &\text{and}\quad a=a_{K,i}=a_{L,j}
    \\[5pt]
    \vphias = \vphisKi \quad  &\text{if}  \quad \ab
    \quad  &\text{and}\quad a=a_{K,i}.
  \end{array}
\end{equation}
This is the same correspondence as in \eqref{eq:loc-basis-RT} between the Raviart-Thomas local basis functions 
$(\vphi_{K,i})$
and the Raviart-Thomas basis functions
$(\vphi_a)_{a\in\T^1}$.
Similarly,
we will denote $\vphiasK = \vphis_{K,i}$ for $a\in\T^1$ such that $a\subset  K$ and $a=a_{K,i}$.
\begin{theorem}
  \label{thm:VF4}
  Consider a family $(\vphi^\star_{K,i})_{K\in\T^2,\,i=1,\, 2,\, 3}$
  of local basis
  functions  on $\Omega$
  that satisfy 
  \begin{equation}
    \label{eq:phi-star-def-1.0}
    \supp \vphi^\star_{K,i} \subset K 
  \end{equation}
  and independently on $i$,
  \begin{equation}
    \label{eq:phi-star-def-1.1}
    \div \vphi^\star_{K,i} = \delta_K, \quad \text{on $K$} .
  \end{equation}
  On $\partial K$, the normal component is given by
  \begin{equation}
    \label{eq:phi-star-def-1.2}
    \vphi^\star_{K,i}\cdot n_K =
    \begin{cases}
      g_{K,i} &\quad  \text{on}\quad a_{K,i} 
      \\
      0       &\quad  \text{otherwise}
    \end{cases}
    ,
  \end{equation}
  where $g_{K,i}$ and $\delta_K$ satisfy 
  equations \eqref{eq:def-g}, \eqref{eq:def-g_Ki} and \eqref{eq:def-delta_K}.\\
  Let  $\drtb$ be constructed from the local basis functions
  $(\vphi^\star_{K,i})_{K,i}$
  with equation \eqref{eq:loc-basis}.
  Then $\drtb$ is a Raviart-Thomas dual basis as defined in definition \ref{def:RT}. 
  Moreover, the coefficients $\ca$ only depend on the mesh $\T$ geometry,
  \begin{equation}
    \label{eq:coef-cotan}
    \begin{array}{lll}
      \ai \quad &\Rightarrow \quad &
      \ca = \left(\cotan \theta_{a,K}
        +\cotan \theta_{a,L}\right) / 2,
      \\[5pt]
      \ab \quad& \Rightarrow \quad &
      \ca = \cotan \theta_{a,K} /2\,.
    \end{array}
  \end{equation}
\end{theorem}
Notations are recalled on figure \ref{fig:cell-2}.
We will also denote $g_{a,K} = g_{K,i}$ for $a\in\T^1$ such that $a\subset  K$ and $a=a_{K,i}$.
\begin{figure}[!ht]
  \begin{center}
    {\includegraphics[height=4.6cm] {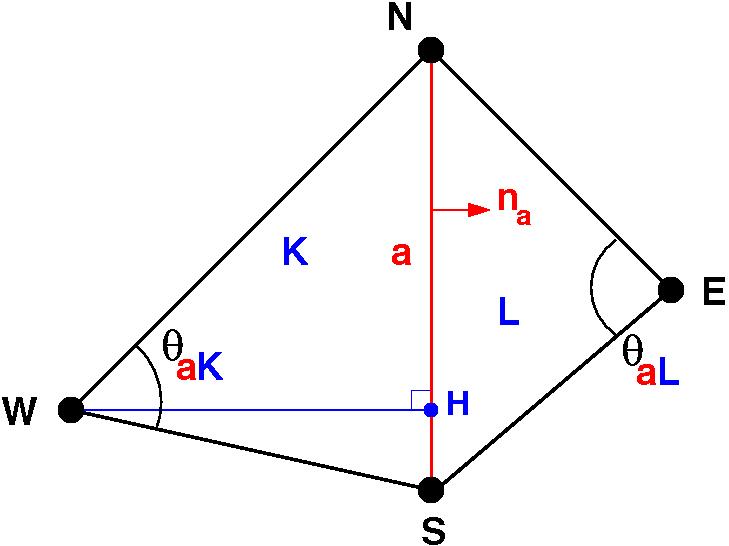}}
  \end{center}
  \caption{Co-boundary of the edge $ \, a \in \T^1 $.}
  \label{fig:cell-2}
\end{figure}
\begin{corollary}
  \label{coro:1}
    Assume that the mesh
    satisfies the Delaunay condition: for all
    internal edge $a\in\T^1$ we have the angle condition 
    $\theta_{a,K} + \theta_{a,L} < \pi$ (denoting $\partial ^c a=(K,L)$).
    Also assume that for any boundary edge $a$,  $\theta_{a,K} < \pi/2$
    (denoting $\partial ^c a=(K)$).
    \\
    Then with \eqref{eq:coef-cotan}, $\ca>0$ and proposition
    \ref{prop:PG-disc} 
    ensures the existence and uniqueness  of the solution
    to the discrete problem.
  Moreover, 
  the mixed Petrov-Galerkin discrete problem
  \eqref{eq:FV-laplace} for the Laplace equation \eqref{eq:laplace}
  coincides with the four point finite volume scheme
    defined and analysed in Herbin \cite {Herbin95}.
  \end{corollary}
Therefore, the  Raviart-Thomas dual basis
does not need to be constructed.
Whatever are $\delta_K$ and $g$ that satisfy equations \eqref{eq:def-g}, 
\eqref{eq:def-g_Ki} and \eqref{eq:def-delta_K}, the coefficients $\ca$ will be unchanged. 
They only depend on the mesh geometry and are given by equation \eqref{eq:coef-cotan}. 
Practically, this means that neither the $\drtb$ nor $\delta_K$ and $g$ need to be computed.
Such a dual basis will be explicitly
computed in section \ref{sec:def-drtb}.
The numerical scheme will always coincide with the 
four point volume scheme. 
Finally, this theorem
provides a new point of view 
for the understanding and analysis of finite volume methods.
\\
Theorem \ref{thm:VF4} gives sufficient conditions in order to build 
Raviart-Thomas dual basis. 
In the sequel we will focus on such Raviart-Thomas dual basis, though more general ones may exist: this will not be discussed in this paper.
  \begin{proof}[Proof of corollary \ref{coro:1}]
    We have the general formula
    $\cotan \theta_1 + \cotan \theta_2 =
    \sin(\theta_1 + \theta_2) / (\sin \theta_1 \,\sin \theta_2)$
    that ensures that $ \ca>0$ under the assumptions in the corollary.
    \\
    For the equivalence between the two schemes,
    it suffices to prove that
    $\cotan \theta_{a,K} /2 = d_{a,K}/|a|$
    where $d_{a,K}$ denotes the distance between the edge $a$
    and the circumcircle centre $C$ of $K$.
    Denote $S$ and $N$ the two vertexes of $a$.
    Then the angle $\widehat{SCN}=2\theta_{a,K}$.
    The distance $d_{a,K}$ is equal to $CH$
    with $H$ the orthogonal projection of $C$ on $a$.
    The triangle $SCN$ being isosceles, $H$ is also  
    the midle of $[SN]$.
    In the right angled triangle $SCH$ we have
    $\widehat{SCH} = \widehat{SCN} /2 = \theta_{a,K}$
    and $\cotan \widehat{SCH} = CH / SH = d_{a,K}/ (|a|/2)$
    which gives the result.
  \end{proof}
\begin{proof}[Proof of theorem \ref{thm:VF4}]
  Consider as in theorem \ref{thm:VF4} a family $(\vphi^\star_{K,i})_{K\in\T^2,\,i=1,\, 2,\, 3}$
  that satisfy,
  \eqref{eq:phi-star-def-1.0}, 
  \eqref{eq:phi-star-def-1.1} and 
  \eqref{eq:phi-star-def-1.2}
  for $\delta_K$ and $g_{K,i}$ such that the assumptions 
  (\ref{eq:def-g}), (\ref{eq:def-g_Ki}) and (\ref{eq:def-delta_K})
  are true.
  Let  $\drtb$ be constructed from the local basis functions
  $(\vphi^\star_{K,i})_{K,i}$
  with equation \eqref{eq:loc-basis}.
  \\ \indent
  Let us first prove that $\drtb$ is a Raviart-Thomas dual basis as in definition \ref{def:RT}.
  Consider an internal edge $a\in\T^1$, $a=(K\vert L)$. With (\ref{eq:phi-star-def-1.2}), 
  we have $\supp \vphias = K\cup L$ and 
  relation (\ref{eq:RT*-2}) holds. 
  With \eqref{eq:loc-basis}, 
  $\vphi^\star_{a\,\vert  K}=\vphi^\star_{K,a}\in H(\div,K)$,     
  $\vphi^\star_{a\,\vert  L}=-\vphi^\star_{L,a}\in H(\div,L)$. 
  The normal flux $\vphias\cdot  n_a$ is continuous across $a=K\cap L$ since $g_{K,a}  =  g_{L,a} $ 
  and with  (\ref{eq:phi-star-def-1.2}). Moreover, $\vphias\cdot n=0$ on the boundary 
  of $K\cup L$ due to (\ref{eq:phi-star-def-1.2}). Therefore $\vphias$ belongs to
  $H(\div,\Omega)$. With formula (\ref{eq:coef-cotan}) and 
  the angle condition made in theorem \ref{thm:VF4}, $(\vphia,\vphias)_0\neq 0$ 
  and so (\ref{eq:RT*-0}) holds.
  \\
  Consider two distinct edges $a,b\in\T^1$. 
  If $a$ and $b$ are not two edges of a same triangle $K\in\T^2$, 
  then $\vphias$ and $\vphi_b$ have distinct supports so that $(\vphias,\vphi_b)_0 = 0$. 
  If $a$ and $b$ are two edges of $K\in\T^2$, then 
  $(\vphias,\vphi_b)_0 = \int_K \vphias\cdot \vphi_b \,\, {\rm d}x$. 
  With the definition (\ref{eq:phi_K-i}) of the local $RT$ basis functions and using the Green formula,
  \begin{align*}
    \pm 4 \vert  K\vert  (\vphias,\vphi_b)_0 &=
    - \int_K \div \vphias \, \vert x - W_{K,b} \vert ^2 \, {\rm d}x        + 
    \int_{\partial  K} \vphias\cdot n \vert s - W_{K,b} \vert ^2 \, {\rm d}s    \\ & =
    - \int_K \delta_K \, \vert x - W_{K,b} \vert ^2 \, {\rm d}x    + 
    \int_0^1 g(s) \,s  ^2 \,ds,
  \end{align*}
  using (\ref{eq:phi-star-def-1.1}), (\ref{eq:phi-star-def-1.2}) and the fact 
  that $W_{K,b}$ is opposite to $b$ and so is a vertex of $a$.
  This implies the orthogonality condition (\ref{eq:RT*}) with the assumptions 
  in  (\ref{eq:def-g}) and (\ref{eq:def-delta_K}).
  \\
  It remains to prove (\ref{eq:RT*-3}).
  In the case where  $a,b\in\T^1$ are two distinct edges, $\int_b \vphias\cdot n_b \, {\rm d}s=0$.
  Assume that $a\in\T^1$ is an edge of $K\in\T^2$. We have $n_a = \varepsilon n_{K,a}$ 
  with $\varepsilon=\pm 1$.
  With relation (\ref{eq:phi-star-def-1.2}) and the divergence formula,
  \begin{displaymath}
    \int_a \vphias\cdot n_a \,\, {\rm d}s  =  \int_{a} 
    (\varepsilon \vphi_{K,a}^\star) \cdot (\varepsilon n_{K,a})\, \, {\rm d}s  
    = \int_{\partial  K} \vphi_{K,a}^\star\cdot  n\, \, {\rm d}s = \int_K \div \vphi_{K,a}^\star\, \, {\rm d}x. 
  \end{displaymath}
  This ensures that $\int_a \vphias\cdot n_a\,ds = 1$ with relation (\ref{eq:phi-star-def-1.1}) 
  and the first assumption in (\ref{eq:def-delta_K}).
  We successively proved (\ref{eq:RT*-0}), (\ref{eq:RT*}), (\ref{eq:RT*-2}) and (\ref{eq:RT*-3}) 
  and then $\drtb$ is a Raviart-Thomas dual basis.
  \\ \indent
  Let us now prove \eqref{eq:coef-cotan}.
  Let $a\in\T^1$ an internal edge with the notations in figure \ref{fig:cell-2}.
  The Raviart-Thomas basis function $\vphia$ has its support in $K\cup L$,
  so that
  $$
  (\vphias,\vphia)_0= \int_K  \vphias \cdot  \vphia \, {\rm d}x + 
  \int_L  \vphias \cdot  \vphia \, {\rm d}x.$$
  With the local decompositions 
  (\ref{eq:loc-basis-RT}) and (\ref{eq:loc-basis}) we have,
  \begin{displaymath}
    (\vphias,\vphia)_0= \int_K \vphiasK\cdot \vphiaK \,\, {\rm d}x + 
    \int_L \vphiasL\cdot \vphiaL \,\, {\rm d}x\,.
  \end{displaymath}
  By  relation \eqref{eq:phi_K-i}, $W$ being the opposite vertex to the edge
  $a$ in the triangle $K$, 
  \begin{eqnarray*}
    4|K| \int_K \vphiasK\cdot\vphiaK \,\, {\rm d}x & =&\int_K \vphiasK\,\nabla|x-W|^2\, {\rm d}x 
    \\
    & =&
    -\int_K \div \vphiasK \,|x-W|^2\, {\rm d}x
    +\int_{\partial  K}\vphiasK\cdot n_K\,|x-W|^2 \, {\rm d}\sigma.
  \end{eqnarray*}
  By hypothesis (\ref{eq:phi-star-def-1.1}) and (\ref{eq:phi-star-def-1.2}),
  and using (\ref{eq:def-delta_K}),
  \begin{align*}
    4|K| \int_K \vphiasK\cdot\vphiaK \,\, {\rm d}x  &=\int_K \delta_K |x-W|^2\, {\rm d}x 
    + \int_a g_{K,a}\,|x-W|^2 \, {\rm d}\sigma
    =\int_a g_{K,a}\,|x-W|^2 \, {\rm d}\sigma.  
  \end{align*}
  Let H be the orthogonal projection of the point $W$ on the edge $a$.
  We have $|x-W|^2= WH^2+|x-H|^2$ and 
  with (\ref{eq:def-g}) and (\ref{eq:def-g_Ki}), 
  $
  \int_a g_{K,a}  \, {\rm d}\sigma = |a| \int_0^1 g(s)/|a| \, {\rm d}s = 1
  $ and so,
  
  \smallskip \centerline { $ \displaystyle  
    4|K|\int_K \vphiasK\cdot\vphiaK \,\, {\rm d}x=WH^2 +  \int_a g_{K,a} |x-H|^2  \, {\rm d}\sigma. $}

  \smallskip\noindent 
  Let $s$ and $s^\star$ respectively be the curvilinear coordinates of $x$ and $H$ on $a$ with origin $S$, then
  
  \smallskip \centerline { $ \displaystyle  
    4|K|\int_K \vphiasK\cdot\vphiaK \,\, {\rm d}x=
    WH^2+|a|^2\int_0^1(s^\star-s)^2g(s)ds.$}

  \smallskip\noindent 
  The assumptions in (\ref{eq:def-g}) on $g$ imply that $2\int_0^1 g(s)sds=1$.
  By expanding $(s^\star-s)^2= s^2 - 2ss^\star + s^{\star \, 2}$ we get,
  $\int_0^1(s^\star-s)^2g(s)ds = s^{\star \, 2} -  s^\star$.
  It follows that,
  \begin{align*}
    4|K|\int_K \vphiasK\cdot\vphiaK \,\, {\rm d}x&=
    WH^2 + 
    \bigl( |a| s^\star \bigl) 
    \bigl( |a| (s^\star-1)\bigl)
    \\&=
    WH^2 + \overrightarrow{SH}\cdot
    \overrightarrow{NH}
    \\&=
    \overrightarrow{WS}\cdot \overrightarrow{WN}
    .
  \end{align*}
  Some trigonometry results in $K$ leads to
  $\sin \theta_{K,a}=\frac{2|K|}{WS\cdot WN}$. As a result, 
  $$ 4|K|\int_K \vphiasK\cdot\vphiaK 
  \, {\rm d}x = 2|K| \, \cotan\theta_{K,a},$$ this gives (\ref{eq:coef-cotan}).
\end{proof}
\section{Stability and convergence}
\label{sec:stab-conv}
In this section we develop a specific choice of dual basis functions.
We  provide for that choice  technical estimates  and
prove a theorem of stability and convergence.
With theorem \ref{thm:VF4}, this leads to an error estimate for
the four point finite volume scheme.
We begin with the main result in theorem \ref{thm:conv-stab}.
Theorem \ref{th:stabilite} provides a methodology in order to get the
inf-sup stability conditions.
The inf-sup  conditions need technical results that are proved
in  subsections \ref{LocaldualRTmassmatrix} to \ref{LocalRTMassMatrix}.
We will need the following angle condition. 
\\
\textbf{Angle assumption.}
Let $\theta_\star$ and $ \theta^\star$ chosen such that
\begin{equation}
  \label{eq:angle-cond0}
  0 < \theta_\star < \theta^\star <\pi /2
\end{equation}
We consider meshes $\T$ such that all the angles
of the mesh are bounded from below and above by
$\theta_\star$ and  $\theta^\star$ respectively:
\begin{equation}
  \label{eq:angle-cond}
  \forall ~ \theta \in \T^{-1},\quad 
  \theta_\star \le \theta \le \theta^\star.
\end{equation}
With that angle condition, the coefficients $(\vphia,\vphias)$ in (\ref{eq:coef-cotan}) 
are strictly positive.
With proposition \ref{prop:PG-disc} this ensures the existence and uniqueness for the 
solution $(u_\T, p_\T)$ 
of the mixed Petrov-Galerkin discrete problem (\ref{eq:PG-laplace}).
\begin{theorem}[Error estimates]
  \label{thm:conv-stab}
  We assume that $\Omega\subset \R^2$ is a bounded polygonal convex domain
  and that $f\in {\rm H}^1(\Omega)$.
  Under the angle hypotheses \eqref{eq:angle-cond0} and
  \eqref{eq:angle-cond},
  there exists a constant $C$ independent
  on $\T$ satisfying \eqref{eq:angle-cond} and independent on $f$
   so that the solution $(u_\T, p_\T)$ of the mixed
  Petrov-Galerkin discrete problem (\ref{eq:PG-laplace}) satisfies,
  \begin{displaymath}
    \Vert u_{\T} \Vert_0+\Vert p_{\T}\Vert_{\Hdiv}\leq 
    C\Vert f\Vert_{0}.
  \end{displaymath}
  Let $u$ be the exact solution to problem \eqref{eq:laplace} and  $p=\nabla  u$
  the gradient, the following error estimates  holds,
  \begin{equation}
    \label{eq:final-error}
    \Vert u-u_{\T} \Vert_0+\Vert p-p_{\T}\Vert_{\Hdiv}\leq 
    Ch_{\T}\Vert f\Vert_{1},
  \end{equation}
  with $h_\T$ the maximal size of the edges of the mesh.
\end{theorem}
\begin{proof}
  We  prove that the unique solution of the mixed Petrov-Galerkin 
  (\ref{eq:PG-laplace}) 
  continuously depends  on the data $f$.
  The bilinear form $Z$ defined in \eqref{eqdef-Z} is continuous,
  with a continuity constant $M$ independent on the mesh $\T$,
  \begin{displaymath}
    | Z(\xi,\eta)  |  \,\,\leq\,\, M \,\, \Vert \xi \Vert_{{\rm L}^2 \times {\rm H}_{\div}}\, 
    \Vert \eta \Vert_{{\rm L}^2 \times {\rm H}_{\div}}\,,\qquad \forall \, \xi\in V,\,\eta\in V^{\star}. 
  \end{displaymath}
  The following uniform inf-sup stability condition:
  there exists a constant $\beta>0$ independent on $\T$ such that,
  \begin{equation}
    \label{inf-sup-unif}
    \forall~ \xi \in V, \, {\rm so \,  that} \,\, \Vert \xi\Vert _{{\rm L}^2 \times {\rm H}_{\div}} = 1 ,\,  
    \exists~ \eta\in V^\star , \,\, \Vert \eta\Vert _{{\rm L}^2 \times {\rm H}_{\div}} \, \le \, 1 \,\, 
    {\rm and} \,\, Z(\xi,\eta)  \, \geq \,   \beta,
  \end{equation}
  is proven in   theorem \ref{th:stabilite} under some conditions.
  Moreover, the two spaces $V$ and $V^\star$ have the same dimension.
  Then the Babu\v{s}ka theorem in \cite{Ba71}, 
  also valid for Petrov-Galerkin
  mixed formulation, applies.
  The unique solution $\xi_{_{\cal T}}=(u_{_\T},p_{_\T})$ of the  discrete scheme 
  \eqref{eq:PG-laplace} satisfies the error estimates,
  and 
  \begin{displaymath}
    \Vert \xi - \xi_{_{\cal T}} \Vert_{{\rm L}^2 \times {\rm H}_{\div}} \,\,\,\leq\,\,\, \Big( 1+\frac{M}{\beta}
    \Bigr)
    \inf_{\zeta\in 
      V }
    \Vert \xi - \zeta \Vert_{{\rm L}^2 \times {\rm H}_{\div}} \,,
  \end{displaymath}
  with $\xi=(u,p)$, $u$ the exact solution to the Poisson problem (\ref{eq:laplace}) and $p=\nabla  u$.
  In our case, this formulation is equivalent to
  \begin{equation}  
    \Vert u - u_{_{\cal T}}  \Vert _{0} \,+\,  \Vert p - p_{_{\cal T}}  \Vert_{\Hdiv}  
    \,\leq\,  C \,\Big( \inf_{ v \in P^0} \, 
    \Vert u-v\Vert_0 +\inf_{q\in RT}\Vert p-q \Vert_{\Hdiv} \Big)
  \end{equation}
  for a constant $C=1+\frac{M}{\beta}$ dependent of $\T$ only through the lowest
  and the highest angles $\theta_{\star}$ and $\theta^{\star}$.
  With the interpolation operators $\Pi_0:{\rm L}^2(\Omega)\rightarrow P^0$
  and $\Pi_{RT}:{\rm H}^1(\Omega)^2\rightarrow RT^0$
  \begin{displaymath}
    \Vert \, u - u_{_{\cal T}}  \, \Vert_{0} \,+\,  \Vert \, 
    p - p_{_{\cal T}}  \, \Vert_{\Hdiv}
    \,\leq\,  C \, \big(  \Vert \, u - \Pi_0   u \, \Vert_{0} 
    + \Vert \, p - \Pi_{\text{\tiny RT}}   p \, \Vert_{\Hdiv}  \big) .
  \end{displaymath}
    On the other hand we have the following interpolation errors:
    \begin{displaymath}
      \Vert \, u - \Pi_0   u \, \Vert_{0} 
      \, \leq \, C_1 h_{\cal T} \, \Vert u \, \Vert_{1}  , \,\quad
      \Vert \, p - \Pi_{\text{\tiny RT}}   p
      \, \Vert_{0} \, \leq C_2\, h_{\cal T} \, \Vert p \, \Vert _{1}  , \, \quad
      \Vert  \div \big( p - \Pi_{\text{\tiny RT}}   p \big)  \Vert_{0} 
      \leq  C_1 h_{\cal T}  \, \Vert \div p \, \Vert _{1}. 
    \end{displaymath}
    On the left, we have the Poincar\'e-Wirtinger inequality where the
    constant $C_1=1/\pi$ is independent on the mesh, due to \cite{payne-weinberger1960}.
    The third inequality is the same as the first one since
    $\Pi_0 \div p = \div \Pi_{\text{\tiny RT}}   p$.
    For the second inequality, the constant $C_2$
    has been proven  in \cite{acosta-duran-1999}
    to be dominated by $1/\sin \theta^\star$ with $\theta^\star$  the maximal angle of the mesh.
  \\
  Then ,
  \begin{displaymath}
    \Vert \, u - u_{_{\cal T}}  \, \Vert _{0} \,+\,  
    \Vert \, p - p_{_{\cal T}}  \, \Vert_{\Hdiv}  
    \,\leq\,  C \,  h_{\cal T} \,  \big( \,  \Vert u \, \Vert  _{1}
    + \Vert p \, \Vert  _{1} + \Vert \div p \, \Vert  _{0} \,  \big) \, ,
  \end{displaymath}
  with a constant $C$ only depending on the maximal angle
    $\theta^\star$.
  Since
  $ - \Delta u = f \,\,\, {\rm in} \, \Omega $,
  with $f\in {\rm H}^1(\Omega)$ and $\Omega$ convex, then
  $u\in {\rm H}^2(\Omega)$ and 
  $\Vert u\Vert_2 \leq c \Vert f\Vert_0$.
  Moreover $p=\nabla u$ and $\div p =-f$ leads to 
  \begin{displaymath}
    \Vert \, u - u_{_{\cal T}}  \, \Vert _{0} \,+\,  \Vert \, p - p_{_{\cal T}}  \, \Vert_{\Hdiv}  
    \,\leq\,  C \,  h_{\cal T} \,  \big( 2 \Vert f \, \Vert_0
    + \Vert f \, \Vert  _{1}   \big) \, .
  \end{displaymath}
  Finally, we get 
  \begin{displaymath}
    \Vert \, u - u_{_{\T}}  \, \Vert _{0} \,+\,  \Vert \, p - p_{_{\cal T}}  \, \Vert_{\Hdiv}  
    \,\leq\,  C \,  h_{\T}  \, \Vert  \, f \, \Vert_1  \, ,
  \end{displaymath}
  that is exactly \eqref{eq:final-error}.
\end{proof}
\begin{theorem} [Abstract stability conditions]
  \label{th:stabilite}
  Assume that the projection $\Pi:~RT \rightarrow  RT^\star,$
  such that $\Pi \vphia= \vphias$ in diagram (\ref{eq:grad-discret}) satisfies, for any $p\in RT$:
  \begin{align}
    \label{eq:stab-cond-1}
    \tag{H1}
    (p,\Pi p)_0 \,&\geq\, A  \, \Vert p\Vert_0^2 ,
    \\[5pt]
    \label{eq:stab-cond-2}
    \tag{H2}
    \Vert \Pi p \Vert_{0}&\leq B  \,  \Vert p\Vert_0,
    \\[5pt]
    \label{eq:stab-cond-3}
    \tag{H3}
    (\div p, \div \Pi p)_0 \,&\geq\, C  \,  \Vert \div  p\Vert_0^2,
    \\[5pt]
    \label{eq:stab-cond-4}
    \tag{H4}
    \Vert \div \Pi p\Vert_0 \,&\leq\,  D \, \Vert \div  p\Vert_0
  \end{align}
  where $A, \,B,\, C,\,D>0$ are constants independent on $\T$.  
  Then the uniform discrete inf-sup condition (\ref{inf-sup-unif})
  holds:
  there exists a constant $\beta>0$ independent on $\T$ such that,
  \begin{displaymath}
    \forall~ \xi \in V, \, {\rm so \,  that} \,\, \Vert \xi\Vert _{{\rm L}^2 \times {\rm H}_{\div}} = 1 ,\,  
    \exists~ \eta\in V^\star , \,\, \Vert \eta\Vert _{{\rm L}^2 \times {\rm H}_{\div}} \, \le \, 1 \,\, 
    {\rm and} \,\, Z(\xi,\eta)  \, \geq \,   \beta.  
  \end{displaymath}
\end{theorem}
This result has been proposed 
by Dubois  in \cite{D02}. For the completeness 
of this contribution,
the proof (presented in the preprint  \cite{Du10})  
is detailed  in Annex A. 

In order to prove the conditions \eqref{eq:stab-cond-1},
\eqref{eq:stab-cond-2},  \eqref{eq:stab-cond-3} and \eqref{eq:stab-cond-4},
one needs some technical lemmas on some estimations of the
dual basis functions so that theorem \ref{th:stabilite} holds.
It is the goal of the next subsections.

\subsection{A specific Raviart-Thomas dual basis}
\label{sec:def-drtb}
\subsubsection*{Choice of the divergence}
For $K$ a given triangle of $\T^2$, we propose
a choice for the divergence $\delta_K$ of the dual basis functions
$\vphisKi,\, 1\leq i\leq 3$ in
\eqref{eq:phi-star-def-1.1}.
We know from \eqref{eq:def-delta_K}
that this function has to be ${\rm L}^2(K)$-orthogonal to the three
following functions:
$ |x-W_{K,i}|^2$ for $i=$1, 2, 3
and that its integral over $K$ is equal to 1. 
We propose to choose $ \delta_K$ as the solution
of the least-square problem:
{\it minimise
  $\int_K \delta_K^2 \,\, {\rm d}x$
  with the constraints in \eqref{eq:def-delta_K}.}
It is well-known 
that the solution belongs to the four dimensional space
$E_K=\Span \left( \un_K, |x-W_{K,i}|^2,~1\le i\le 3\right)$
and is obtained by the inversion of an appropriate Gram matrix.

\begin{lemma}
  \label{lem:bound-delta_K}
  For the above construction of $\delta_K$, we have the following estimation:
  \begin{displaymath}  
    |K|  \int_K  \delta_K ^2 \dd x  \leq   \nu,
    \qquad  \text{with} \quad 
    \nu=\frac{8 \cdot 3^5 \cdot 23}{5}\,\,  {\frac{1}{\tan ^4\theta_\star}} .
  \end{displaymath} 
\end{lemma}
The proof of this result is  technical and has been obtained
with the help of a formal calculus software. It  is detailed in Annex C.
\\[10pt]
\subsubsection*{Choice of the flux on the boundary of the triangle}
A continuous function $g:(0,1)\rightarrow \R$ satisfying the conditions
\eqref{eq:def-g} 
can be chosen  as the following polynomial:
\begin{equation}
  \label{eq:g}
  g(s) = 30 s \, (s-1) \, (21 s^2- 21 s + 4). 
\end{equation}
\subsubsection*{Construction of the Raviart-Thomas dual basis}
\begin{figure}[!ht]
  \begin{center}
    \includegraphics[height=4cm] {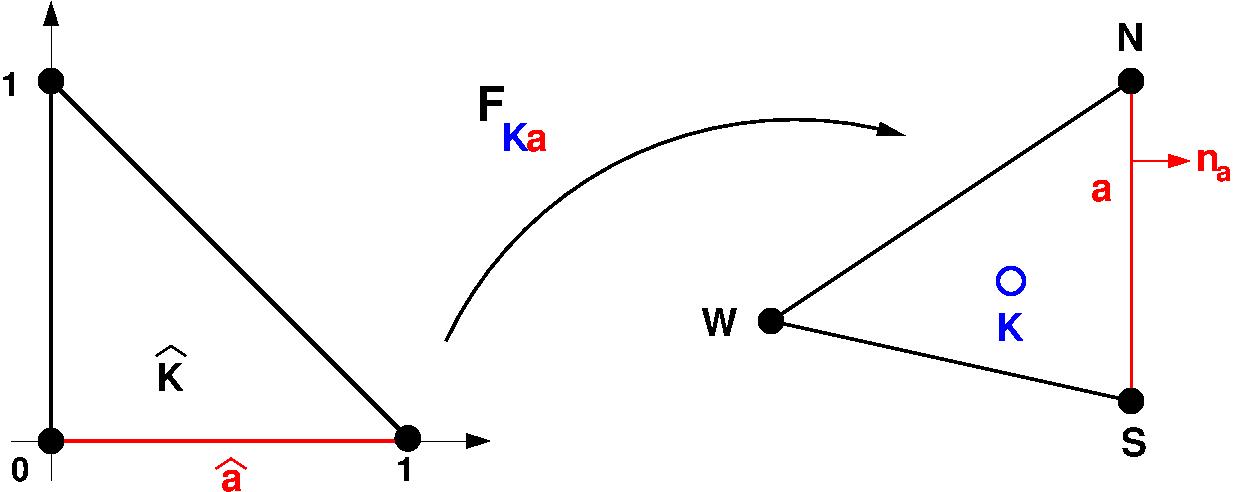}
  \end{center}
  \caption{Affine mapping  $ \, F_{K,a} \, $ between the reference
    triangle $ \, \hat K \, $  and the given triangle $ \, K $. }
  \label{fig:cell-3}
\end{figure}
For a triangle $K$ and an edge $a$ of $K$, we construct now a possible
choice of the dual function
$\vphiasK$ satisfying
\eqref{eq:phi-star-def-1.0}, \eqref{eq:phi-star-def-1.1}
and \eqref{eq:phi-star-def-1.2}.
Let $ \, F_{K,a} \,$ be an  affine function  that maps the reference triangle
$\hat K$ into the triangle $ \, K \, $ such that 
the edge $ \, \hat a \equiv [0,1]\times\{0\} \, $
is transformed into the given edge 
$ \, a \subset \partial K$. Then the mapping 
$ \, \,  \widehat{K} \ni \widehat{x} \longmapsto x =  F_{K,a}(\widehat{x}) \in K $
is one to one. 
We define  $ x=F_{K,a}(\hat x)$ for any $\hat x \in \hat K$ and the 
right hand side 
$ \,\, \widetilde \delta_K(\hat x) = 2 \, |K| \, \delta_K(x)$.
With $g$ defined in \eqref{eq:g},
let us define $ \, \widehat g\in {\rm H}^{1/2}(\partial \hat K) \, $
according to

\smallskip \smallskip \centerline { $ \displaystyle  
  \widehat g:=\left\lbrace
    \begin{array}{ll}
      g \text{ on } \hat a=[0,1]\times\{0\}\\
      0 \text{ elsewhere on }  \partial \hat K\,.
    \end{array}
  \right. $} 

\smallskip \noindent 
Since \quad $ \displaystyle \int_{\widehat{K}}
\widetilde {\delta_K}  \,\, {\rm d}x \,= 1
\, = \, \int_{\partial \hat{K}}    \widehat{g} \,\, {\rm d}\gamma   $,
the  inhomogeneous Neumann problem
\begin{equation}
  \label{eq:zeta-Neumann}
  \Delta \zeta_K= \widetilde \delta_K  \,\, \text{ in } \hat K \,, \quad 
  \frac{\partial \zeta_K}{\partial n} = \widehat g \,\,   \text{ on } \partial \hat K 
\end{equation}
is well posed.
The dual function $ \, \varphi_{K,a}^\star \, $ is defined according to 
\begin{equation}
  \label{eq:defduale}
  \displaystyle 
  \vphiasK (x) = {\frac{1}{ \det ( {\rm d}F_{K,a} ) }} \, {\rm d}F_{K,a}
  \,
  \,
  \widehat {\nabla} \zeta_{K}  \, . 
\end{equation} 
These so-defined functions satisfy the hypotheses
\eqref{eq:phi-star-def-1.0}, \eqref{eq:phi-star-def-1.1} and
\eqref{eq:phi-star-def-1.2}
of theorem \ref{thm:VF4}.
Let us now  estimate their ${\rm L}^2$-norm.

\subsubsection*{${\rm L}^2$-norm of the Raviart-Thomas  dual basis }
\label{LocaldualRTmassmatrix}
An upper bound on the ${\rm L}^2$ norm of the Raviart-Thomas  dual basis
will be needed in order to prove the stability conditions in theorem \ref{th:stabilite}. 
This bound is given in lemma \ref{lem:gram-RT-star}.
It only involves the mesh minimal angle $\theta_\star$.
\begin{lemma}
  \label{lem:L2-estimation-RT-star}
  For $K\in \T^2$ and $a\in \T^1$, $ \, a  \subset \partial K$, we have

  \smallskip \centerline { $ \displaystyle  
    \Vert \varphi  _{K,\, a}^\star \Vert_{0 \, K}\, \leq \, \mu^\star  $} 

  \smallskip \noindent 
  where $\mu^\star$ is essentially a function of
  the smallest angle $\theta_{\star}$
  of the triangulation.
\end{lemma}
\begin{proof}
  Since the reference triangle $ \, \hat K \, $ is convex
  and  $ \, \widehat g \in {\rm H}^{1/2}(\partial \widehat{K}) $, the 
  solution  $  \zeta_{K} $  of the  Neumann problem \eqref{eq:zeta-Neumann}
  satisfies the regularity property  (see for example \cite{ADN59})
  $ \, \zeta_{K} \in {\rm H}^2 (\widehat{K}) $, 
  continuously to the data:

  \smallskip \centerline { $ \displaystyle  
    \displaystyle \Vert \zeta_{K} \Vert_{2,\hat K} \, \leq \, C_{\widehat{K}} \, 
    \Big( \,   \Vert \widetilde{ \delta_K } \Vert _{0,\widehat{K}} 
    +  \Vert \widehat{g}  \Vert_{1/2 ,\, \partial \widehat{K}} \, \Big)  \, .  $} 

  \smallskip \noindent 
  Moreover thanks to lemma \ref{lem:bound-delta_K},

  \smallskip \centerline { $ \displaystyle  
    \Vert \widetilde{ \delta_K } \Vert_{0,\hat K}^2   
    = \int_{\widehat{K}}  \widetilde{ \delta_K }^2 \,  {\rm d}  \widehat{x} 
    = \int_K \, (2|K|\delta_K)^2 \, \frac{1}{ \det ( {\rm d}F_{K,a} ) } \dd x
    = 2 \, |K|\int_K \delta_K^2  \, {\rm d}x \leq \,2\, \nu  $} 

  \smallskip \noindent 
  and then

  \smallskip \centerline { $ \displaystyle  
    \Vert \widehat {\nabla } \zeta_{K} \Vert_{0,\hat K} \, \leq \,
    C_{\widehat{K}} \, 
    \Big( \,  \sqrt{2\nu} +  \Vert \widehat{g}  \Vert_{1/2 ,\, \partial \widehat{K}} \, \Big)  \, .   $} 

  \smallskip \noindent 
  Since the dual function $\vphiasK$
  is defined by \eqref{eq:defduale} and
  $\Vert {\rm d} F_{K,a}\Vert ^2 \leq \frac{8|K|}{\sin \theta_{\star}}$ from direct geometrical computations on the triangle $K$, we obtain
  \begin{displaymath} 
   \Vert \vphiasK \Vert_{0, K}^2   \leq  
      \displaystyle \Big( {\frac{1}{2 \, |K|}} \Big)^2  \, \Big( {\frac{8 \, |K|}{\sin \, \theta_\star }} \Big)  \, 
      \Vert \widehat {\nabla } \zeta_{K} \Vert _{0,\hat K}^2 \, ( 2 \, |K| )\,.
  \end{displaymath} 
  Then
  $\displaystyle  \Vert \vphiasK  \Vert_{0, K}^2     \leq   \displaystyle  (\mu^\star) ^2\,,$
  with $\displaystyle (\mu^\star) ^2=
  {\frac{4}{\sin \, \theta_\star }} \, \, C_{\widehat{K}} ^2\, 
  \Big( \,  \sqrt{2\nu} +  \Vert \widehat{g}  \Vert_{1/2 ,\, \partial \widehat{K}} \, \Big)^2 \, . $
\end{proof}

\begin{lemma}
  \label{lem:gram-RT-star}
  For $K\in \T^2$ and $q\in RT^\star$:
  \begin{displaymath}
    \Vert  \Pi q \Vert_{0,K}^2
    \le 3(\mu^{\star})^2 \sum_{i=1}^3 q_{K,i}^2  \,\,.    
  \end{displaymath}
\end{lemma}
\begin{proof}
  We have for a triangle $K$,
  $\di\Pi q= \sum_{i=1}^3 q_{K,i}\vphi^\star_{K,i} $,
    and so
    \begin{equation*}
      \Vert  \Pi q \Vert_{0,K}^2
      \le
      \Big(
      \sum_{i=1}^3 |q_{K,i}|\, \Vert \vphi^\star_{K,i}\Vert _{0,K}
      \Big)^2
      \le
      \sum_{i=1}^3  |q_{K,i}|^2
      \sum_{i=1}^3  \Vert \vphi^\star_{K,i}\Vert _{0,K}^2
       \end{equation*}
    Then lemma \ref{lem:L2-estimation-RT-star} applies and:
    $\di
    \Vert  \Pi q \Vert_{0,K}^2
       \leq 3(\mu^\star)^2\sum_{i=1}^3 q_{K,i}^2\,.
    $
  \end{proof}
\subsection{Local Raviart-Thomas  mass matrix} %
\label{LocalRTMassMatrix}
The proof of the stability conditions in theorem  \ref{th:stabilite} involves
lower and upper bounds of the eigenvalues of the local Raviart-Thomas
mass matrix. We will need the following result proved in Annex B.

\begin{lemma}
  \label{lem:gram-RT}
  For $p\in RT$ and $K\in \T^2$:

  \smallskip \centerline { $ \displaystyle  
    \lambda_\star \,\sum_{i=1}^3 p_{K,i}^2\,\le\,
    \Vert  p \Vert_{0,K}^2\,
    \le \,\lambda^\star \,\sum_{i=1}^3 p_{K,i}^2, $} 

  \smallskip \noindent 
  for two constants $\lambda_\star$ and $\lambda^\star$ only depending
  on $\theta_{\star}$ in (\ref{eq:angle-cond0}),

  \smallskip \centerline { $ \displaystyle  
    \lambda_\star=   \frac{\tan ^2\theta_{\star}}{48},\quad
    \lambda^\star=\frac{5}{4 \tan \theta_{\star}}. $} 
\end{lemma}
\subsection{The hypotheses of theorem \ref{th:stabilite} are satisfied}
\label{sec:preuve_thm}

Let us finally prove  that the conditions
\eqref{eq:stab-cond-1}, \eqref{eq:stab-cond-2},  \eqref{eq:stab-cond-3} 
and \eqref{eq:stab-cond-4} of theorem \ref{th:stabilite} hold.
The proof relies on lemma \ref{lem:gram-RT}, lemma \ref{lem:gram-RT-star} and lemma \ref{lem:bound-delta_K}
involving the mesh independent constants
$\lambda_\star$, $\lambda^\star$, $\mu^\star$ and $\nu$.
In the following, 
$p$ denotes an element of $RT$ and $K$ a fixed mesh triangle.
It is recalled that on $K$,  
$ \di p=\sum_{i=1}^3 p_{K,i} \,\, \vphi_{K,i}$.

\bigskip   \noindent 
{\bf Condition \eqref{eq:stab-cond-1}.}
Using the orthogonality property \eqref{eq:RT*}, and
relation \eqref{eq:coef-cotan} successively, leads to
\begin{displaymath}
  (\Pi p,p)_{0,K} = \sum_{i=1}^3 p_{K,i}^2 \, (\vphi^\star_{K,i},\vphi_{K,i})_{0,K}  
  =\frac{1}{2}\sum_{i=1}^3 p_{K,i}^2 \,
  \cotan \theta_{K,i}\geq\frac{1}{2}\cotan \theta^{\star} \sum_{1=1}^3 p_{K,i}^2.
\end{displaymath}
Lemma \ref{lem:gram-RT} gives a lower bound,

\smallskip  \centerline { $ \displaystyle 
  (\Pi p,p)_{0,K} \geq \frac{\cotan \theta^\star}{ 2\lambda^\star}
  \Vert p\Vert^2_{0,K}.$}

\smallskip \noindent 
Summation over all $K\in\T^2$ gives \eqref{eq:stab-cond-1}
with,

\smallskip  \centerline { $ \displaystyle 
  A=\frac{\cotan \theta^\star}{ 2\lambda^\star}=\frac{2}{5}\cotan \theta^\star\tan\theta_\star. $}

\bigskip   \noindent 
{\bf Condition \eqref{eq:stab-cond-2}.}
Using successively lemma \ref{lem:gram-RT-star} and lemma \ref{lem:gram-RT} we get,

\smallskip  \centerline { $ \displaystyle 
  \Vert \Pi p \Vert_{0,K}^2 \le 
  3 (\mu^\star)^2 \sum_{i=1}^3 p_{K,i}^2 
  \le \frac{3(\mu^\star)^2}{\lambda_\star} \Vert p \Vert_{0,K}^2. $}

\smallskip \noindent 
With the values of $\lambda_\star$  given in
lemma \ref{lem:gram-RT-star}  this implies \eqref{eq:stab-cond-2} with,

\smallskip \centerline { $ \displaystyle  
  B = \sqrt{\frac{3(\mu^\star)^2}{\lambda_\star}}=\frac{12}{ \tan \theta_\star}\mu^\star.  $}
\smallskip \noindent 
\bigskip   \noindent 
{\bf Condition \eqref{eq:stab-cond-3}.}
Relation \eqref{eq:RT*-3.2} induces 
$(\div \Pi p, \div p)_{0,K} = \Vert \div p \Vert^2_{0,K}$ since $\div p$ is a constant on $K$,
and as a result inequality \eqref{eq:stab-cond-3} indeed is an equality with 
$$C=1.$$
\bigskip   \noindent 
{\bf Condition \eqref{eq:stab-cond-4}.}
With equation (\ref{eq:RT-div}) we get 
$\Vert  \div p \Vert ^2_{0,K}= \left(  \sum_{1=1}^3 p_{K,i} \right)^2/|K|$ 
and with  condition   \eqref{eq:phi-star-def-1.1}, 
$\div \Pi p = \delta_K(x) \sum_{1=1}^3 p_{K,i}$. 
Therefore we get,

\smallskip \centerline { $ \displaystyle  
  \Vert  \div \Pi p \Vert ^2_{0,K} = \int_K \delta_K^2 \dd x ~\left(  \sum_{1=1}^3 p_{K,i} \right)^2 = |K| \int_K \delta_K^2 \dd x ~\Vert  \div p \Vert ^2_{0,K}.  $}

\smallskip \noindent 
Condition \eqref{eq:stab-cond-4} follows from lemma \ref{lem:bound-delta_K},   with 

\smallskip \centerline { $ \displaystyle  
  D=\sqrt{\nu},\quad    
  \nu=\frac{8 ~3^5 ~23}{5}\,\,  {\frac{1}{\tan ^4\theta_\star}}.  $}

\smallskip \noindent 

\section*{Conclusion}
In this contribution we present a way to
define a local discrete gradient of a piecewise constant function on a triangular mesh. 
This discrete gradient is obtained from a Petrov-Galerkin formulation and belongs to the Raviart-Thomas function space of low order.
We have defined suitable
 dual test functions 
of the Raviart-Thomas basis functions.
For the  Poisson problem, we can interpret the   Petrov-Galerkin  formulation as a finite volume method.
Specific constraints for the  dual test functions enforce stability. 
Then the convergence can be established with the usual methods of mixed finite elements.
It would be interesting to try to extend this work in several directions:
the three-dimensional case,
the case of general diffusion problems and also the case of higher degree finite element methods.

\begin{appendices}
\renewcommand{\appendixname}{Annex}

\section{Proof of theorem \ref{th:stabilite}}
\label{sec:annex-1}

In this section, we consider meshes $\T$
that satisfy
the angle conditions (\ref{eq:angle-cond})  
parametrised by the pair $ \, 0 < \theta_\star < \theta^\star < {\frac{\pi}{2}} \,$.
We suppose that the interpolation operator $\Pi$
defined in section 1 by
$\,  \Pi: RT \longrightarrow RT^\star \, $ with $\Pi \vphia=\vphias$
satisfies the following properties:
there exist four positive constants 
 $A$, $B$, $C$ and $D$ only depending on $\theta_\star$ and $\theta^\star$
such that for all $ q \in RT$
\begin{eqnarray}
  \label{hyp-01} 
  (q,\Pi q) &\,\geq\,& A \,\,  \Vert q\Vert_0^2 
  \,,  
  \\[5pt]
  \label{hyp-02} 
  \Vert \Pi q \Vert_{0}&\leq& B  \,  \Vert q\Vert_0 \, ,
  \\[5pt]
  \label{hyp-03} 
  (\div q, \div \Pi q)_0 &\geq& C  \,  \Vert \div  q\Vert_0^2 
  \, , 
  \\[5pt]
  \label{hyp-04}   
  \Vert \div \Pi q\Vert_0 &\leq&  D \, \Vert \div  q\Vert_0 \, .
\end{eqnarray}
Let us first prove  the following proposition relative to the lifting of scalar fields.
\begin{proposition} [Divergence lifting of scalar fields]
  Under the previous hypotheses (\ref{hyp-01}), (\ref{hyp-02}), (\ref{hyp-03}) and (\ref{hyp-04}), 
  there exists some strictly
  positive constant $\, F \,$ that only depends of the 
  minimal and maximal angles $\theta_\star$ and $\theta^\star$
  such that for any mesh   $\, {\T} \,$ 
   and for any scalar field $\, u \,$ constant in each element $\, K $ of $\T$,  
   ($u \, \in P^0$), 
  there exists some vector field $\,\, q \, \in  RT^\star ,\,$ 
  such that 
  \moneq  \label{estimation-norme-hdiv} 
  \parallel q \parallel_{H_{\div}} \,\,\, \leq \,\,\, F \, \parallel u \parallel _{0} 
  \monend 
  \moneq  \label{minoration-produit-scalaire} 
  (\, u \,,\, {\rm div} \, q \, )_0 \quad \,  \geq \,\, \,\, \parallel u \parallel _{0} ^2\, \,. \, 
  \monend   
\end{proposition}

\begin{proof}  
    Let $\,\, u  \in \, P^0 \, $ be a discrete scalar function
  supposed to be   constant in each triangle $\, K \,$ of the mesh $\,  {\T} . \,$
  Let $\, \psi \, \in {\rm H}^1_0(\Omega) \,$ be the variational solution of the Poisson
  problem 
  \moneq  \label{poisson} 
  \Delta \, \psi \,\, = \,\, u \quad  {\rm in} \,\, \Omega \,\,,\quad 
  \psi \,=\, 0 \quad   {\rm on} \,\, \partial \Omega \,.\,
  \monend 
  Since $\, \Omega \,$ is convex, the solution $\, \psi \,$ of the problem (\ref{poisson}) belongs
  to the space $\, {\rm H}^2(\Omega) \,$ and there exists some constant $\, G > 0 \,$ that only
  depends on $\, \Omega \,$ such that 
  \begin{equation*}
     \parallel \psi \parallel_{2} \,\,\, \leq \,\,\, G \, \parallel u \parallel_{0}  \,. 
   \end{equation*}
  Then the field $\,\, \nabla \psi \,\,$ belongs to the space $\, {\rm H}^1(\Omega) \, \times
  \,  {\rm H}^1(\Omega) .\, $  It is in consequence possible to interpolate this field in a
  continuous way (see {\it e.g.}  Roberts and Thomas \cite{RT91})  in the space $\, 
  {\rm H}({\rm div},\, \Omega) \,\,$ with the help of the fluxes on the edges: 
   \begin{equation*}
    p_a \,\,= \,\, \int_{\displaystyle a} \, {\frac{\partial \psi}{\partial n_a}} \,
    {\rm d} \gamma \,,\qquad p\,=\, \sum_{a \, \in  \T^1}  \, p_a \,\,
    \varphi_a \, \,\, \in\,\, RT  \, . 
  \end{equation*}
  Then  there exists a constant $\, L > 0 \,$
  such that 
  \moneq  \label{controle-norme-hdiv} 
  \parallel p \parallel_{H_{\div}} \,\,\, \leq \,\,\,L \,\, \parallel
  u \parallel _{0} \,. 
  \monend 
  The two fields $\, {\rm div} \, p \,$ and $\, u \,$ are constant in each
  element $\, K \,$ of the mesh $\, {\T} .\,$  Moreover, we have:
  \moneqstar   
  \int_K \, {\rm div} \, p \,\,{\rm d}x \,\,= \,\, \int_{\partial K} \, p \, \cdot  \, n \,\,
  {\rm d}\gamma \,\,=\,\,  \int_{\partial K} \,{\frac{\partial \psi}{\partial n}} \,\,
  {\rm d}\gamma \,\,=\,\, \int_K \,\Delta \psi \,\, {\rm d}x \,\,= \,\, \int_K \, u
  \,\, {\rm d}x \,. 
  \monendstar 
  Then we have exactly,  
  $ \, {\rm div} \, p \, = \, u \,$ in $ \, \Omega \, $ 
  because this relation  is a consequence of the above property for the mean values. 
  \\
  Let now $\, \, \Pi \, p \,$ be the interpolate of $p$ in the ``dual space''  
  $ \, RT^\star \, $ and $q \,\,= \,\, \frac{1}{C} \, \Pi \, p $,
  \moneqstar   
  q \,\,= \,\, {\frac{1}{C}} \, \Pi \, p 
  \,\,= \,\, {\frac{1}{C}} \, \sum_{a \in \T^1} p_a \, \varphi_a^{\star} \quad {\rm with} \quad
  \Pi \, p  \,= \,  \sum_{a \in \T^1} p_a \, \varphi_a^{\star}  .\, 
  \monendstar 
  We have as a consequence of (\ref{hyp-03})  and $ \, {\rm div} \, p =  u $
  that,
  \moneqstar 
  ( \, u \,,\, {\rm div} \, q \,)_0 \,\,= \,\, {\frac{1}{C}} \, ( \, {\rm div} \, p \,,\,
  {\rm div} \,  \Pi \, p \,) \,\, \geq \,\, \parallel {\rm div} \, p \parallel_{0} ^2
  \,\,\,= \,\,\, \parallel u \parallel_{0} ^2 
  \monendstar 
  that establishes (\ref{minoration-produit-scalaire}).
  Moreover, we have due to equations \eqref{hyp-02}, \eqref{hyp-04} and \eqref{controle-norme-hdiv}:
  \begin{eqnarray*}
    \parallel q \parallel_{0} &\,=\,& {\frac{1}{C}} \,\parallel   \Pi \, p
  \parallel_{0} \, \leq \, {\frac{B}{C}} \,\parallel    p \parallel_{0}
  \, \leq \, {\frac{B L }{C}} \,\parallel  u \parallel_{0}\,,
  \\
  \parallel {\rm div} \, q \parallel_{0} &\, = \,& {\frac{1}{C}} \,\parallel
  {\rm div} \, \Pi \, p \parallel_{0} \, \leq \, {\frac{D}{C}} \,\parallel
  {\rm div} \, p \parallel_{0}  \, = \,  {\frac{D}{C}} \,\parallel  u
  \parallel_{0} \,.
  \end{eqnarray*}
  Then the two above inequalities establish 
  the estimate \eqref{estimation-norme-hdiv} with 
  $\,\, F \,=\,  {\frac{1}{C}}  \,\sqrt{ B^2 {\rm L}^2 + D^2} \,\,$ and
  the proposition is proven. 
\end{proof}
\subsubsection*{Proof of  theorem \ref{th:stabilite}}

\noindent 
We suppose that the dual
Raviart-Thomas basis satisfies the Hypothesis (\ref{hyp-01}) to (\ref{hyp-04}). 
We introduce the constant $ \, F > 0 \,$ such that (\ref{estimation-norme-hdiv}) 
and  (\ref{minoration-produit-scalaire}) are realised for some vector field
$ \, \widetilde{q} \in RT^{\star} \, $ for any $ \, u \in P^0 $:
\moneq  \label{q-tilda} 
\parallel \widetilde{q} \parallel_{H_{\div}} \,\,\, \leq \,\,\, F \, \parallel u \parallel _{0} 
\quad  {\rm and} \quad  
(\, u \,,\, {\rm div} \, \widetilde{q} \, )_0   \,\, \geq \,\, \parallel u \parallel _{0} ^2\, \,. 
\monend   
\monitem We set $\, \displaystyle a = {\frac{1}{2}} \big( \sqrt{4+F^2} - F \big) , \, $
$\, \displaystyle b = {\frac{A}{D + \sqrt{B^2 + D^2}}} \, $ with the constants $ \, F $, 
$\, A $, $\, B \, $ and $\, D \, $ introduced in (\ref{q-tilda}),  (\ref{hyp-01}), 
(\ref{hyp-02}) and (\ref{hyp-04}) respectively.
We shall prove that for
\moneq  \label{beta-infsup} 
\beta = {\frac{b \, a^2}{1 + 2 \, a \, b}} \, ,
\monend   
the inf-sup condition 
\moneq  \label{inf-sup-unif-2}  \left\{ \begin{array}{ll}
    \exists \, \beta   > \,0 \,,\,\,\, 
    \,,\,\, \forall \, \xi \, \in P^0\times RT 
    \,\, $ such that $ \displaystyle  \,\,
    \,\, \parallel \xi \parallel_{{\rm L}^2 \times {\rm H}_{\div}}  \,\,=\,\, 1 \,,\,\, \\
    \exists \, \eta \, \in  \, P^0\times RT^\star  
    \,,\,\,  \parallel \eta \parallel_{{\rm L}^2 \times {\rm H}_{\div}}  \,\,\leq\,\, 1 \,\,\, {\rm and} \,\,\, 
    Z(\xi,\,\eta) \,\, \geq \,\, \beta \, 
  \end{array} \right. \monend
is satisfied.
We set 
\moneq  \label{alpha-infsup} 
\alpha \equiv  a - \beta = a {\frac{1 + a \, b}{1 + 2 \, a \, b}} > 0 \, .
\monend   
Then we have after an elementary algebra:  $ \,\, \, a \, F + a^2 = 1 . \, $ In consequence, 
\moneq  \label{tecno-1-infsup} 
(\alpha + \beta) \, F + \alpha^2 + \beta^2 \,\leq \, 1 
\monend   
because 
$ \displaystyle \, (\alpha + \beta) \, F + \alpha^2 + \beta^2 \, 
\leq \,  (\alpha + \beta) \, F + (\alpha + \beta)^2 = 1 $. Moreover, 
\moneq  \label{tecno-2-infsup} 
\beta \leq b \, \alpha^2 
\monend   
thanks to the relations (\ref{beta-infsup}) and (\ref{alpha-infsup}):
\begin{equation*}
  \beta -  b \, \alpha^2  = {\frac{1}{(1 + 2 \, a \, b)^2}} \, \Big[ 
b\, a^2 \, (1 + 2 \, a \, b) - b\, a^2 \, (1 +  a \, b)^2  \Big] = - {\frac{a^4 \, b^3}{(1 + 2 \, a \, b)^2}}.
\end{equation*}
\monitem
Consider now 
$\,\,  \xi \equiv (u,\,p) \,\,$ satisfying the hypothesis of unity norm in the product space:
\moneq  \label{xi-norme-1} 
\parallel \xi   \parallel_{{\rm L}^2 \times {\rm H}_{\div}} \,\,\,\,\equiv \,\,\, \parallel u
\parallel\ib{0}^2 \,\,+\,\, \parallel p \parallel\ib{0}^2 \,\,+\,\, \parallel {\rm
  div}\,p  \parallel\ib{0}^2  \,\,\, =\,\,\, 1  \,.\,
\monend   
Then at last one of these terms is not too small and due to the three terms that arise
in relation (\ref{xi-norme-1}), the proof is divided into three parts. 

\smallskip \noindent {\it (i)} 
If the  condition 
$ \displaystyle  \,\, 
\parallel {\rm div}\, p   \parallel\ib{0} \,\,  \geq \,\, \beta \,\,$ 
is satisfied, we set  
\moneqstar   
v \,\,= \,\, {\frac{{\rm div}\, p}{\parallel {\rm div}\, p  \parallel\ib{0}}} \, 
\,,\quad q \,\,=\,\, 0 \,\,,\qquad \eta \,\,=\,\, (v,\,q) \, . 
\monendstar 
Then,
$\,\, \, \displaystyle  \parallel {\rm div}\, v  \parallel\ib{0} \,=\, 1 \,\,
\,\,$ and   $\,\, \, \displaystyle  \parallel \eta  \parallel\ib{0} \,\leq \, 1
\,.\, $ Moreover
\moneqstar 
Z(\xi,\, \eta) \,=\, ( \, {\rm div} \, p \,,\, v \,)_0 \,\,= \,\, 
\parallel {\rm div}\, p  \parallel\ib{0} \,\, \geq \,\, \beta 
\monendstar 
and the relation (\ref{inf-sup-unif-2})  is satisfied in this particular case. 

\smallskip \noindent {\it (ii)} 
If the  conditions 
$ \displaystyle  \,\, 
\parallel {\rm div}\, p   \parallel\ib{0} \,\,  \leq \,\, \beta \,\,$ and 
$ \displaystyle  
\parallel  p   \parallel\ib{0} \,\,  \geq \,\, \alpha \,\,$ 
are  satisfied, we set  
\moneqstar   
v \,\,= \,\, 0 \,\,\, \,,\quad q \,\,=\,\,  {\frac{1}{\sqrt{B^2+D^2}}} \,\, \Pi \, p 
\,\,,\qquad \eta \,\,=\,\, (v,\,q) \,. 
\monendstar 
We check that $\Vert  \eta \Vert _{{\rm L}^2 \times {\rm H}_{\div}} \le 1$:
\begin{eqnarray*}
\Vert  \eta \Vert _{{\rm L}^2 \times {\rm H}_{\div}}^2 &\,=\,& \parallel q   \parallel\ib{0}^2 
+ \parallel {\rm div} \, q  \parallel\ib{0}^2  \,\, \leq \,\, 
{\frac{1}{B^2+D^2}} \, \Big( B^2 \, \parallel p   \parallel\ib{0}^2 
+ D^2 \,  \parallel {\rm div} \, p  \parallel\ib{0}^2  \Big) 
\\[5pt]  
&\,\leq \,&    \parallel p   \parallel\ib{0}^2 +  \parallel {\rm div} \, p  \parallel\ib{0}^2
\, \leq \, \Vert  \xi \Vert _{{\rm L}^2 \times {\rm H}_{\div}}^2 \,\, = \,\, 1 .
\end{eqnarray*}
Then
\begin{equation*}
Z(\xi, \, \eta) = (p, \, q)_0 + ( u , \, {\rm div} \,q )_0  
\geq   {\frac{1}{\sqrt{B^2+D^2}}} \,  \Big( (p, \, \Pi \, p )_0 - \parallel u  \parallel\ib{0} \, 
\parallel {\rm div}\, \Pi \, p   \parallel\ib{0}  \Big) \,.  
\end{equation*}
Moreover $ \, \displaystyle \parallel u  \parallel\ib{0} \leq 1 $, then\\
\begin{equation*}
Z(\xi, \, \eta) \geq   {\frac{1}{\sqrt{B^2+D^2}}} \, \Big( A  \parallel p  \parallel\ib{0}^2 \,-\,  D \, 
\parallel {\rm div}\, p   \parallel\ib{0}  \Big) 
\geq   {\frac{1}{\sqrt{B^2+D^2}}} \, \Big( A  \parallel p  \parallel\ib{0}^2 \,-\,  D \, \beta  \Big)  
\,\, \geq  \,\, \beta   
\end{equation*}
because the inequality 
$ \,\,  \displaystyle \big( D + \sqrt{B^2+D^2} \big) \beta \, \leq \, A \, \alpha^2 \,\, $ 
is exactly the inequality (\ref{tecno-2-infsup}). 
Then the relation (\ref{inf-sup-unif-2})  is satisfied in this second case.

\smallskip \noindent {\it (iii)} 
If the  last conditions 
$ \displaystyle  \,\, 
\parallel {\rm div}\, p   \parallel\ib{0} \, \leq \, \beta \,\,$ and 
$ \displaystyle  
\parallel  p   \parallel\ib{0} \, \leq \, \alpha \,$ 
are  satisfied, we first remark that the first component $ \, u \, $ has a 
norm bounded below: 
from (\ref{tecno-1-infsup}),
\moneqstar   
0 \, < \, a \, F \, = \, (\alpha + \beta) \, F \, \leq \, 1 - \alpha^2 - \beta^2 
\leq \, 1 - \parallel  p   \parallel\ib{0}^2 - \parallel {\rm div}\, p   \parallel\ib{0}^2 
\,\, = \,\,  \parallel u  \parallel\ib{0}^2 \, . 
\monendstar 
Then we set,  
\moneqstar   
v \,\,= \,\, 0 \,\,\, \,,\quad q \,\,=\,\,  {\frac{1}{F}} \,\, \widetilde{q}  
\,\,,\qquad \eta \,\,=\,\, (v,\,q) \, ,
\monendstar 
with a discrete vector field $ \, \widetilde{q} \,$ satisfying the inequalities 
(\ref{q-tilda}). Then,
\begin{eqnarray*}
Z(\xi, \, \eta) &=& ( u , \, {\rm div} \,q )_0  + (p, \, q)_0 
= {\frac{1}{F}} \, \big( ( u , \, {\rm div} \, \widetilde{q} )_0  + 
(p, \,  \widetilde{q})_0 \big) 
\\[5pt]
&\geq&   {\frac{1}{F}} \,  \parallel u  \parallel\ib{0}^2  \,-\,
{\frac{1}{F}} \, \parallel  p   \parallel\ib{0} \, \parallel  \widetilde{q} \parallel_{H_{\div}}  
\\[5pt]
&\geq&   {\frac{1}{F}} \,  \parallel u  \parallel\ib{0}^2  \,-\,
\alpha \,  \parallel u  \parallel\ib{0}  \,\quad \text{ due to }\quad (\ref{q-tilda})
\\[5pt]
&\geq&  \beta,
\end{eqnarray*}
because, due to (\ref{tecno-1-infsup}) we have the following  inequalities:
\moneqstar   
\parallel u  \parallel\ib{0} \, \alpha + \beta \,\, \leq \,\, \alpha + \beta 
\,\, \leq \,\,   {\frac{1}{F}} \, \big( 1 - \alpha^2 - \beta^2 \big) 
\,\, \leq \,\,   {\frac{1}{F}} \,\parallel u  \parallel\ib{0}^2 \, . 
\monendstar 
Then the relation (\ref{inf-sup-unif-2})  is satisfied in this third  case
and the proof is completed. 
$\hfill $   $\blacksquare$

%
\section{Proof of lemma \ref{lem:gram-RT}} 
We first recall the statement of lemma \ref{lem:gram-RT}.
\begin{lemme-4}
  For $p\in RT$ and $K\in \T^2$:
  \begin{displaymath}    
    \lambda_\star \sum_{i=1}^3 p_{K,i}^2\le
    \Vert  p \Vert_{0,K}^2
    \le \lambda^\star \sum_{i=1}^3 p_{K,i}^2, 
  \end{displaymath}
  for two constants $\lambda_\star$ and $\lambda^\star$ only depending on $\theta_{\star}$ 
  in (\ref{eq:angle-cond0}),
  \begin{displaymath}    
    \lambda_\star=   \frac{\tan ^2\theta_{\star}}{48},\quad
    \lambda^\star=\frac{5}{4 \tan \theta_\star}. 
  \end{displaymath}
\end{lemme-4}

\noindent 
The following technical result will be necessary for the proof of lemma \ref{lem:gram-RT}.
\begin{lemma}
  \label{lem:rhoK-K}
  The gyration radius of a triangle $K$ is defined as,
  $\rho_K^2=\frac{1}{|K|}\int_K|X-G|^2$,
  with $G$ the barycentre of the triangle $K$.
  It satisfies,
  \begin{displaymath}    
    \frac{1}{6} \,\leq \, \frac{\rho_K^2}{|K|} \,\leq \, {\frac{1}{3\, \tan \theta_{\star}}}. 
  \end{displaymath}    
\end{lemma}
\begin{proof}
  Let $A_i$ and $a_i$, $1$=1, 2, 3, be respectively
  the three vertices and edges of the triangle $K$.
  One can check that:
  $
  36 \, \rho_K^2 =\sum_{i=1}^3|A_iA_{i+1}|^2 = \sum_{i=1}^3|a_i|^2.
  $

  On the one hand,
  $|K|\leq \frac{1}{2}|A_iA_j||A_iA_k|\leq\frac{1}{4}\big( |A_iA_j|^2+|A_iA_k|^2\big)$ for any
  $1\leq i, j, k \leq 3$ and $i\neq j$, $i\neq k$  and $k\neq j$.
  Then $3|K|\leq \frac{1}{2}\sum_{i=1}^3|A_iA_{i+1}|^2 = 18 \, \rho_K^2$,
  that gives the lower bound.

  On the other hand,
  using the definition of the tangent, 
  $|K|\geq \frac{1}{4} |a_i|^2\tan \theta_\star $, for $1\le i\le 3$.
  Then
  $
  3|K|\geq \frac{1}{4}\tan \theta_\star\sum_{1=1}^3|a_i|^2
  = 9 \bigl(\tan \theta_\star \bigl)\rho_K^2,
  $
  that gives the upper bound.
\end{proof}
\subsubsection*{Proof of lemma \ref{lem:gram-RT}}
For a triangle $K$, the local $RT$ mass matrix is $G_K := \left[ (\vphi_{K,i}, \vphi_{K,j})_{0,K}\right]_{1\le i,j\le 3}$.
Explicit computation obtained by Baranger-Maitre-Oudin in \cite{BMO96}
gives some properties on the gyration radius:
\begin{equation}
  \label{cotan-rho}
  \sum_{i=1}^3 \cotan \theta_i=9 \frac{\rho_K^2}{|K|}
\end{equation}
where  $\theta_i$ are the angles of the triangle $K$
and lead to information on the Raviart-Thomas basis as follows:
\begin{equation}
  \label{eq:norme_phi}
  \begin{array}{ll}    
    \di\Vert\vphi_{K,i}\Vert_{0,K}^2=\frac{1}{6}\cotan \theta_i+\frac{3}{4}\frac{\rho_K^2}{|K|}
    \;\\ 
    \di (\vphi_{K,i},\vphi_{K,j})_{0,K}=\frac{1}{4}\frac{\rho_K^2}{|K|}
    -\frac{1}{9}\Bigl(\cotan \theta_i+\cotan \theta_j-\frac{\cotan \theta_k}{2}\Bigl)
    =-\frac{3}{4}\frac{\rho_K^2}{|K|}+\frac{\cotan \theta_k}{6}
  \end{array}
\end{equation}
where $k$ is the third index of the triangle $K$ ($k\neq i,\, j$, $1\leq i,\,j\,,k\leq3$).
\\ \indent
\textbf{Derivation of $\lambda^\star$}.
The triangle $K\in \T^2$ is fixed and $p\in RT$ rewrites
$p=\di\sum_{i=1}^3p_{K,i}\vphi_{K,i}$ on $K$.
One can easily prove that,
$$\Vert  p \Vert^2_{0,K} \leq \tr(G_K)\sum_{i=1}^3p_{K,i}^2
,\; \text{ where}\;
\tr(G_K)=\sum_{i=1}^3 \Vert \vphi_{K,i}\Vert ^2_{0,K}
\text{ is the trace of } G_K
.$$
With the properties \eqref{cotan-rho} and \eqref{eq:norme_phi}, $\tr(G_K)=\frac{15}{4|K|}\rho_K^2$.
This leads to the value of $\lambda^\star$ thanks to lemma \ref{lem:rhoK-K}.
\\

\textbf{Derivation of $\lambda_\star$}.
In order to compute $\lambda_\star$, we want to  find a lower bound
for the smallest 
eigenvalues of 
the Gram matrix $G_K$ . 
The characteristic polynomial is given by
$$P(\lambda)=- \det(\lambda I-G_K)= -[\lambda^3-\tr(G_K)\lambda^2+R\lambda -\det G_K]$$
where $R:=\sum_{i=1}^3R_i$ with 
$R_i:=\Vert\vphi_i\Vert_0^2\Vert\vphi_{i+1}\Vert_0^2-(\vphi_i,\vphi_{i+1})_{0,K}^2$
with the usual notation if $i=3$, $ \vphi_{i+1}=\vphi_1$. 
Since $P(\lambda)$ is of degree 3 with positive roots, 
the smallest root $\lambda_{\star}$ is such that
$\lambda_{\star}\geq \frac{\det(G_K)}{R}$.
As $G_K$ is a Gram matrix, the determinant of $G_K$
is the square of the volume of polytope generated by the basis function:
$$\det(G_K)={\rm vol}(\vphi_1,\vphi_2,\vphi_3)^2.$$
We expand each basis function on the orthogonal basis made of the three
vector fields:
$\overrightarrow{i},\, \overrightarrow{j},\,x-G$.
Then the volume can be computed via a 3 by 3 elementary determinant.
This leads  to
$$\det(G_K)=\frac{\rho_K^2}{16|K|}.$$
The explicit computation of
$R_i$ 
with help of \eqref{eq:norme_phi} leads to
$$R_i=\frac{1}{36}\cotan\theta_i\cotan \theta_{i+1}+
\frac{1}{8}(\cotan\theta_i+\cotan\theta_{i+1})\frac{\rho_K^2}{|K|}
-\frac{\cotan^2\theta_{i+2}}{36}+\frac{\cotan\theta_{i+2}}{4}\frac{\rho_K^2}{|K|}.$$
Using the geometric property that
$\di\sum_{i=1}^3 \cotan\theta_i \,\cotan\theta_{i+1}=1$ and the previous property
\eqref{cotan-rho}
the summation gives
$$R= \sum_{i=1}^3 R_i=\frac{1}{12}+\frac{9}{4}\frac{\rho_K^4}{|K|^2}.$$
Then using lemma \ref{lem:rhoK-K} we get
$\di R \leq \,  \frac{1}{4\, \tan ^2\theta_{\star}} + \frac{1}{12} $
and, one can conclude that

\qquad \qquad \qquad \qquad 
$ \displaystyle  
\lambda_{\star}  \, \geq \, \frac{\tan^2 \theta_{\star}}{8 \, ( \tan ^2\theta_{\star} + 3) }  
\, \geq \,\frac{\tan ^2\theta_{\star}}{48 } \, $ since $\theta_{\star}\leq \frac{\pi}{3}$. 
\hfill
$\blacksquare$

\section{Proof of  lemma \ref{lem:bound-delta_K}}
\label{sec:annex-2}
We express the function $\delta_K$ as a linear combination of the functions
$\un_K$ and  $|x-W_{K,i}|^2$, for $1\le i\le 3 $.
Thanks to the conditions \eqref{eq:def-delta_K}, we solve formally
a 4 by 4 linear system (with the help of a formal calculus software)
in order to explicit the components.
We can then compute the integral $I$ given by,
$$I= |K|\int_K \delta_K^2 \,\, {\rm d}x.$$
The result is a symmetric function of the length $\vert a_i\vert $ of the three edges of the triangle $K$. It is a ratio of two homogeneous polynomials of degree 12. More precisely $I$ reads,
\begin{displaymath}
  \displaystyle I \,=\, \frac{1}{128} \, {\frac{N}{  |K|^4 \, D}}, 
\end{displaymath}
where $N$ and $D$ respectively are homogeneous polynomials of degree 12 and 4.
The exact expressions of $D$ and $N$ are,
\begin{align}
  \label{integrale}
  D &= \frac{7}{4} \, \sigma_4 \, - 
  \,  \frac{1}{2} \, \Sigma_{2 , 2 , 0},
  \\\label{annexC:def-N}
  N &= 9 \, \sigma_{12} - 15 \, \Sigma_{10 , 2 , 0}  + 15 \, \Sigma_{8 , 4 , 0} 
  - 33 \,  \Sigma_{8 , 2 , 2} - 18 \,\Sigma_{6 , 6 , 0} + 48 \, \Sigma_{6 , 4 , 2}
  + 558 \, \varpi^4 ,
\end{align}
with the following definitions,
\moneqstar 
\Sigma_{n, m, p} \,\equiv \, \sum_{i \not=  j \not= k}   \,  |a_i|^n  \,\, |a_j|^m  \,\, |a_k|^p  \,,\,
\qquad \varpi  \,\equiv \, |a_1|  \,  |a_2|  \,  |a_3| \,=\,  \Sigma_{1,1,1} \,,
\monendstar 
and where $ \, \sigma_p \,$ is the sum of  of the three edges length  $\vert  a_j \vert $ to the power $p$:
\begin{displaymath}
  \sigma_p \, \equiv \, \sum_{j=1}^3 \, |a_j|^p  \, .  
\end{displaymath}
The lemma \ref{lem:bound-delta_K} states an upper bound of $I$.
To prove it, we  look for an upper bound of $N$ and a lower bound of $D$.
\\
The denominator $D$ in (\ref{integrale}) is the 
difference of two positive expressions. 
We  remark that, 
\moneqstar
\sigma_2^2 = \big( a_1^2 + a_2^2 + a_3^2 \big)^2 = \sigma_4 + 2 \,  \Sigma_{2 , 2 , 0}. 
\monendstar 
We have on the one hand,
\moneq 
\sigma_4 = \sigma_2^2   - 2 \,  \Sigma_{2 , 2 , 0} \,   ,
\label{sigma4}  \monend
and on the other hand $ \,\, a_i^2 \,\, a_j^2 \, \leq \, \frac{1}{2} \big( a_i^4 + a_j^4 \big)$.
Then by summation
\moneq 
\Sigma_{2 , 2 , 0} \, \leq \, \sigma_4 \, . 
\label{sigma220}  \monend
In the expression of $D$ in (\ref{integrale}), we split the term relative to 
$ \, \sigma_4 \,$ into two parts: 
\moneqstar
D = \alpha \, \sigma_4 + \beta \, \sigma_4  -   \frac{1}{2} \, \Sigma_{2 , 2 , 0}  \,, \quad {\rm with} \quad
\alpha  + \beta = \frac{7}{4} \,. 
\monendstar
Then thanks to (\ref{sigma4}), 

\smallskip \noindent  \qquad 
$  D = \alpha \, \big(  \sigma_2^2   - 2 \,  \Sigma_{2 , 2 , 0} \big)  + \beta \, \sigma_4 
-   \frac{1}{2} \, \Sigma_{2 , 2 , 0} \,=\, \alpha \, \sigma_2^2   + \beta \, \sigma_4  
-  \big( 2 \, \alpha +  \frac{1}{2} \big) \, \Sigma_{2 , 2 , 0}  $

\smallskip \noindent  \qquad \quad 
$ \geq \,\, \alpha \, \sigma_2^2 
+  \big[ \beta - \big( 2 \, \alpha +  \frac{1}{2} \big) \big] \, \Sigma_{2 , 2 , 0}  \,  $
\qquad   due to (\ref{sigma220}).

\smallskip  \noindent
We force the relation $ \, \beta - \big( 2 \, \alpha +  \frac{1}{2} \big) = 0 .\,$ 
Then $ \, 3 \, \beta =  \frac{7}{2} + \frac{1}{2} = 4 \, $ and $ \, \alpha =  \frac{7}{4} - \frac{4}{3}  
= \frac{5}{12} > 0  $.  We deduce  the lower bound,
\moneq 
D \, \geq \, \frac{5}{12} \, \sigma_2^2 \, . 
\label{minor-D}  \monend
We give now an upper bound of the numerator $N$ given in (\ref{annexC:def-N}).
We remark  that the expression $ \, \sigma_2^3 \equiv  \big( a_1^2 + a_2^2 + a_3^2 \big)^3 \,  $ 
contains 27 terms.
After an elementary calculus we obtain, 
\moneq 
\sigma_2^3  = \sigma_6 + 3 \, \Sigma_{4 , 2 , 0} + 6 \, \varpi^2 \, . 
\label{sigma2-p3}  \monend
In an analogous way, 
\moneq 
\sigma_4^3  = \sigma_{12} + 3 \, \Sigma_{8 , 4 , 0} + 6 \, \varpi^4 \, . 
\label{sigma4-p3}  \monend

\smallskip \noindent We can now bound the numerator $N$:

\smallskip \noindent \qquad  
$ N \, \leq \, 9 \, \sigma_{12}  + 15 \, \Sigma_{8 , 4 , 0}  + 48 \,   \Sigma_{6 , 4 , 2}
+ 558 \, \varpi^4 $ 

\smallskip \noindent \qquad \quad  $ \,\,\,=\,
4  \, \sigma_{12} + 5 \, \big(  \sigma_{12}  + 3 \, \Sigma_{8 , 4 , 0} + 6 \, \varpi^4 \big) 
+ 48 \,  \varpi^2 \,   \Sigma_{4 , 2 , 0} +  528 \, \varpi^4 $

\smallskip \noindent \qquad \quad  $ \,\,\,=\,
4  \, \sigma_{12} + 5 \, \sigma_4^3 + 16 \, \varpi^2 \, \big( 3 \, \Sigma_{4 , 2 , 0} + 6 \, \varpi^2 \big) 
+ 432 \, \varpi^4 $ \hfill due to (\ref{sigma4-p3}) \qquad ~

\smallskip \noindent \qquad \quad  $ \,\,\, \leq \,
4  \, \sigma_{12} + 5 \, \sigma_4^3 + 16 \, \varpi^2 \,  \sigma_2^3  + 24  \, \varpi^4 + 408 \, \varpi^4 $
\hfill due to (\ref{sigma2-p3}) \qquad ~

\smallskip \noindent \qquad \quad  $ \,\,\, \leq \,
4  \, ( \sigma_{12} + 6 \,  \varpi^4 ) 
+ 5 \, \sigma_4^3 + \frac{16}{6}  \,  \sigma_2^6    + 408 \, \varpi^4 $
\hfill due to (\ref{sigma2-p3}) \qquad ~

\smallskip \noindent \qquad \quad  $ \,\,\, \leq \,
9 \, \sigma_4^3 + \frac{16}{6}  \,  \sigma_2^6  +  \frac{408}{36}  \,  \sigma_2^6   $ 
\hfill due to (\ref{sigma4}) \qquad ~

\smallskip \noindent \qquad \quad  $ \,\,\, \leq 
\,\, \big( 9 +  \frac{8}{3} +  \frac{34}{3} \big) \, \sigma_2^6   $ 
\hfill due to (\ref{sigma4-p3}) \qquad ~

\smallskip \noindent and finally,
\moneq 
N \, \leq \, 23  \, \sigma_2^6  \, . 
\label{majo-N}  \monend
We observe that the  upper bound (\ref{majo-N})  is clearly not optimal!
We then combine the definition (\ref{integrale}) and inequalities 
(\ref{minor-D})  and (\ref{majo-N}): 
\moneqstar
I \, \leq \,  
\frac{1}{128} \, {\frac{ 23 \, \sigma_2^6}{{\frac{5}{12}} \, \sigma_2^2}} \,  {\frac{1}{  |K|^4 }} 
\,\, \leq \,\, {\frac{ 3 \,.\, 23}{5 \,.\, 32}} \, \Big( {\frac{\sigma_2}{  |K| }} \Big)^4   .\,\, 
\monendstar 
We use that   $
36 \, \rho_K^2 = \sum_{i=1}^3|a_i|^2 =\sigma_2
$ and the lemma \ref{lem:rhoK-K} to get,
\begin{displaymath}
  \dfrac{\sigma_2}{\vert  K\vert } \le
  \dfrac{12}{\tan \theta_\star}.
\end{displaymath}
It follows that
$   \, \,  \displaystyle I \, \leq \,  {\frac{ 3 \,.\, 23 \,.\,12^4}{5 \,.\, 2\,.\,4^2}} \, 
\Bigl(  {\frac{1}{\tan \theta_{\star}}} \Bigl)^4$, so 
ending the proof of lemma \ref{lem:bound-delta_K}.
\hfill $\blacksquare$

\end{appendices}


\bibliographystyle{plain}
\bibliography{biblio}

\end{document}